\documentclass[%
  onecolumn
]{mpi2015-cscpreprint}

\usepackage{algorithm, algpseudocode, algorithmicx}
\usepackage{amsthm} 
    \newtheorem{corollary}{Corollary}
    \newtheorem{definition}{Definition}
    
    \newtheorem{lemma}{Lemma}
    \newtheorem{proposition}{Proposition}
    \newtheorem{remark}{Remark}
    \newtheorem{theorem}{Theorem}
\usepackage{amsmath, amscd, amssymb, bm, mathrsfs, mathtools}
\usepackage{array, booktabs, makecell} 
\newcolumntype{P}[1]{>{\centering\arraybackslash}p{#1}}
\newcolumntype{M}[1]{>{\centering\arraybackslash}m{#1}}
\usepackage[american]{babel}
\usepackage{graphicx}
\usepackage{enumerate}
\usepackage[ddmmyy,hhmmss]{datetime}
\usepackage{url}
\usepackage{setspace}
\usepackage{soul}
\usepackage[textwidth=30mm]{todonotes}
\usepackage{tikz}
\usepackage{ifthen}

\usepackage{textcomp}

\usepackage[mathscr]{euscript}
\let\condrel\relax
\DeclareMathOperator{\condrel}{cond_{rel}}
\let\condabs\relax
\DeclareMathOperator{\condabs}{cond_{abs}}
\let\trace\relax
\DeclareMathOperator{\trace}{trace}
\DeclareMathOperator{\diag}{diag}
\DeclareMathOperator{\blkdiag}{blkdiag}
\DeclareMathOperator{\blkspn}{span^\square}

\DeclareMathOperator{\rnk}{rank}
\DeclareMathOperator{\spec}{spec}	

\renewcommand{\d}{\,\mathrm{d}}		

\renewcommand{\vec}[1]{{\texttt{\upshape vec}}{\left(#1\right)}}
\newcommand{\unvec}[1]{{\texttt{\upshape unvec}}{\left(#1\right)}}
\newcommand{\fold}[1]{{\texttt{\upshape fold}}{\left(#1\right)}}
\newcommand{\unfold}[1]{{\texttt{\upshape unfold}}{\left(#1\right)}}
\newcommand{\bcirc}[1]{{\texttt{\upshape bcirc}}{\left(#1\right)}}

\usepackage{xspace}		
\newcommand{\bfomfom}{\texttt{bfomfom}\xspace}
\newcommand{\bcirccode}{\texttt{bcirc}\xspace}
\newcommand{\dft}{\texttt{dft}\xspace}
\newcommand{\lowrank}{\texttt{low-rank}\xspace}
\newcommand{\tfrechet}{\texttt{t\_frechet}\xspace}

\newcommand{\spC}{\mathbb{C}}				

\newcommand{\spCnn}{\spC^{n \times n}}

\newcommand{\spCtensor}{\spC^{n \times m \times p}}
\newcommand{\spK}{\mathscr{K}}				
\newcommand{\spR}{\mathbb{R}}				

\newcommand{\vd}{\bm{d}}
\newcommand{\ve}{\bm{e}}

\newcommand{\vh}{\bm{h}}
\newcommand{\vk}{\bm{k}}

\newcommand{\vu}{\bm{u}}
\newcommand{\vv}{\bm{v}}

\newcommand{\vw}{\bm{w}}

\newcommand{\vone}{\boldsymbol{1}}

\newcommand{\vC}{\bm{C}}

\newcommand{\vE}{\bm{E}}

\newcommand{\vV}{\bm{V}}

\newcommand{\vW}{\bm{W}}

\renewcommand{\AA}{\mathcal{A}}
\newcommand{\BB}{\mathcal{B}}
\newcommand{\CC}{\mathcal{C}}
\newcommand{\DD}{\mathcal{D}}
\newcommand{\EE}{\mathcal{E}}
\newcommand{\FF}{\mathcal{F}}
\newcommand{\GG}{\mathcal{G}}
\newcommand{\LL}{\mathcal{L}}
\newcommand{\HH}{\mathcal{H}}

\newcommand{\II}{\mathcal{I}}

\newcommand{\KK}{\mathcal{K}}
\newcommand{\MM}{\mathcal{M}}

\renewcommand{\SS}{\mathscr{S}}

\newcommand{\XX}{\mathcal{X}}
\newcommand{\YY}{\mathcal{Y}}
\newcommand{\XXarrow}{\vec{\XX}}

\newcommand{\bVV}{\bm{\mathcal{V}}}
\newcommand{\bWW}{\bm{\mathcal{W}}}

\newcommand{\inv}{{-1}}

\newcommand{\fAABB}{f(\AA)*\BB}

\newcommand{\norm}[1]{\left\lVert#1\right\rVert}

\makeatletter
\newsavebox{\@brx}
\newcommand{\llangle}[1][]{\savebox{\@brx}{\(\m@th{#1\langle}\)}%
	\mathopen{\copy\@brx\kern-0.5\wd\@brx\usebox{\@brx}}}
\newcommand{\rrangle}[1][]{\savebox{\@brx}{\(\m@th{#1\rangle}\)}%
	\mathclose{\copy\@brx\kern-0.5\wd\@brx\usebox{\@brx}}}
\makeatother

\newif\ifcuboidshade
\newif\ifcuboidemphedge

\tikzset{
  cuboid/.is family,
  cuboid,
  shiftx/.initial=0,
  shifty/.initial=0,
  dimx/.initial=3,
  dimy/.initial=3,
  dimz/.initial=3,
  scale/.initial=1,
  densityx/.initial=1,
  densityy/.initial=1,
  densityz/.initial=1,
  rotation/.initial=0,
  anglex/.initial=0,
  angley/.initial=90,
  anglez/.initial=225,
  scalex/.initial=1,
  scaley/.initial=1,
  scalez/.initial=0.5,
  front/.style={draw=black,fill=gray},
  top/.style={draw=black,fill=gray},
  right/.style={draw=black,fill=gray},
  shade/.is if=cuboidshade,
  shadecolordark/.initial=black,
  shadecolorlight/.initial=white,
  shadeopacity/.initial=0.15,
  shadesamples/.initial=16,
  emphedge/.is if=cuboidemphedge,
  emphstyle/.style={thick},
}

\newcommand{\tikzcuboidkey}[1]{\pgfkeysvalueof{/tikz/cuboid/#1}}

\newcommand{\tikzcuboid}[1]{
    \tikzset{cuboid,#1} 
  \pgfmathsetlengthmacro{\vectorxx}{\tikzcuboidkey{scalex}*cos(\tikzcuboidkey{anglex})*28.452756}
  \pgfmathsetlengthmacro{\vectorxy}{\tikzcuboidkey{scalex}*sin(\tikzcuboidkey{anglex})*28.452756}
  \pgfmathsetlengthmacro{\vectoryx}{\tikzcuboidkey{scaley}*cos(\tikzcuboidkey{angley})*28.452756}
  \pgfmathsetlengthmacro{\vectoryy}{\tikzcuboidkey{scaley}*sin(\tikzcuboidkey{angley})*28.452756}
  \pgfmathsetlengthmacro{\vectorzx}{\tikzcuboidkey{scalez}*cos(\tikzcuboidkey{anglez})*28.452756}
  \pgfmathsetlengthmacro{\vectorzy}{\tikzcuboidkey{scalez}*sin(\tikzcuboidkey{anglez})*28.452756}
  \begin{scope}[xshift=\tikzcuboidkey{shiftx}, yshift=\tikzcuboidkey{shifty}, scale=\tikzcuboidkey{scale}, rotate=\tikzcuboidkey{rotation}, x={(\vectorxx,\vectorxy)}, y={(\vectoryx,\vectoryy)}, z={(\vectorzx,\vectorzy)}]
    \pgfmathsetmacro{\steppingx}{1/\tikzcuboidkey{densityx}}
  \pgfmathsetmacro{\steppingy}{1/\tikzcuboidkey{densityy}}
  \pgfmathsetmacro{\steppingz}{1/\tikzcuboidkey{densityz}}
  \newcommand{\dimx}{\tikzcuboidkey{dimx}}
  \newcommand{\dimy}{\tikzcuboidkey{dimy}}
  \newcommand{\dimz}{\tikzcuboidkey{dimz}}
  \pgfmathsetmacro{\secondx}{2*\steppingx}
  \pgfmathsetmacro{\secondy}{2*\steppingy}
  \pgfmathsetmacro{\secondz}{2*\steppingz}
  \ifthenelse{\equal{\dimx}{1}}
    {\foreach \x in {\steppingx,...,\dimx}}
    {\foreach \x in {\steppingx,\secondx,...,\dimx}}
  {     \ifthenelse{\equal{\dimy}{1}}
    {\foreach \y in {\steppingy,...,\dimy}}
    {\foreach \y in {\steppingy,\secondy,...,\dimy}}
    {   \pgfmathsetmacro{\lowx}{(\x-\steppingx)}
      \pgfmathsetmacro{\lowy}{(\y-\steppingy)}
      \filldraw[cuboid/front] (\lowx,\lowy,\dimz) -- (\lowx,\y,\dimz) -- (\x,\y,\dimz) -- (\x,\lowy,\dimz) -- cycle;
    }
    }
    \ifthenelse{\equal{\dimx}{1}}
    {\foreach \x in {\steppingx,...,\dimx}}
    {\foreach \x in {\steppingx,\secondx,...,\dimx}}
  { \ifthenelse{\equal{\dimz}{1}}
    {\foreach \z in {\steppingz,...,\dimz}}
    {\foreach \z in {\steppingz,\secondz,...,\dimz}}
    {   \pgfmathsetmacro{\lowx}{(\x-\steppingx)}
      \pgfmathsetmacro{\lowz}{(\z-\steppingz)}
      \filldraw[cuboid/top] (\lowx,\dimy,\lowz) -- (\lowx,\dimy,\z) -- (\x,\dimy,\z) -- (\x,\dimy,\lowz) -- cycle;
        }
    }
    \ifthenelse{\equal{\dimy}{1}}
    {\foreach \y in {\steppingy,...,\dimy}}
    {\foreach \y in {\steppingy,\secondy,...,\dimy}}
  { \ifthenelse{\equal{\dimz}{1}}
    {\foreach \z in {\steppingz,...,\dimz}}
    {\foreach \z in {\steppingz,\secondz,...,\dimz}}
    {   \pgfmathsetmacro{\lowy}{(\y-\steppingy)}
      \pgfmathsetmacro{\lowz}{(\z-\steppingz)}
      \filldraw[cuboid/right] (\dimx,\lowy,\lowz) -- (\dimx,\lowy,\z) -- (\dimx,\y,\z) -- (\dimx,\y,\lowz) -- cycle;
    }
  }
  \ifcuboidemphedge
    \draw[cuboid/emphstyle] (0,\dimy,0) -- (\dimx,\dimy,0) -- (\dimx,\dimy,\dimz) -- (0,\dimy,\dimz) -- cycle;%
    \draw[cuboid/emphstyle] (0,\dimy,\dimz) -- (0,0,\dimz) -- (\dimx,0,\dimz) -- (\dimx,\dimy,\dimz);%
    \draw[cuboid/emphstyle] (\dimx,\dimy,0) -- (\dimx,0,0) -- (\dimx,0,\dimz);%
    \fi

    \ifcuboidshade
    \pgfmathsetmacro{\cstepx}{\dimx/\tikzcuboidkey{shadesamples}}
    \pgfmathsetmacro{\cstepy}{\dimy/\tikzcuboidkey{shadesamples}}
    \pgfmathsetmacro{\cstepz}{\dimz/\tikzcuboidkey{shadesamples}}
    \foreach \s in {1,...,\tikzcuboidkey{shadesamples}}
    {   \pgfmathsetmacro{\lows}{\s-1}
        \pgfmathsetmacro{\cpercent}{(\lows)/(\tikzcuboidkey{shadesamples}-1)*100}
        \fill[opacity=\tikzcuboidkey{shadeopacity},color=\tikzcuboidkey{shadecolorlight}!\cpercent!\tikzcuboidkey{shadecolordark}] (0,\s*\cstepy,\dimz) -- (\s*\cstepx,\s*\cstepy,\dimz) -- (\s*\cstepx,0,\dimz) -- (\lows*\cstepx,0,\dimz) -- (\lows*\cstepx,\lows*\cstepy,\dimz) -- (0,\lows*\cstepy,\dimz) -- cycle;
        \fill[opacity=\tikzcuboidkey{shadeopacity},color=\tikzcuboidkey{shadecolorlight}!\cpercent!\tikzcuboidkey{shadecolordark}] (0,\dimy,\s*\cstepz) -- (\s*\cstepx,\dimy,\s*\cstepz) -- (\s*\cstepx,\dimy,0) -- (\lows*\cstepx,\dimy,0) -- (\lows*\cstepx,\dimy,\lows*\cstepz) -- (0,\dimy,\lows*\cstepz) -- cycle;
        \fill[opacity=\tikzcuboidkey{shadeopacity},color=\tikzcuboidkey{shadecolorlight}!\cpercent!\tikzcuboidkey{shadecolordark}] (\dimx,0,\s*\cstepz) -- (\dimx,\s*\cstepy,\s*\cstepz) -- (\dimx,\s*\cstepy,0) -- (\dimx,\lows*\cstepy,0) -- (\dimx,\lows*\cstepy,\lows*\cstepz) -- (\dimx,0,\lows*\cstepz) -- cycle;
    }
    \fi 

    \draw (-\dimx,.2*\dimy) node {\scriptsize $n$};
    \draw (-.2*\dimx,-.5*\dimy) node {\scriptsize $m$};
    \draw (-.55*\dimx,.85*\dimy) node {\scriptsize $p$};

  \end{scope}
}

\makeatother

\begin{document}


\title{The Fr\'echet derivative of the tensor t-function}
  
\author[$\ast$]{Kathryn Lund}
\affil[$\ast$]{Computational Methods in Systems and Control Theory, Max Planck Institute for Dynamics of Complex Technical Systems, Sandtorstr.\ 1, Magdeburg, 39106, Germany.\authorcr
  \email{lund@mpi-magdeburg.mpg.de}, \orcid{0000-0001-9851-6061}}
  
\author[$\dagger$]{Marcel Schweitzer}
\affil[$\dagger$]{School of Mathematics and Natural Sciences, Bergische Universit\"at Wuppertal, 42097 Wuppertal, Germany.\authorcr
  \email{marcel@uni-wuppertal.de}, \orcid{0000-0002-4937-2855}}
  
\shorttitle{The tensor t-Fr\'{e}chet derivative}
\shortauthor{K. Lund, M. Schweitzer}
\shortdate{}
  

  
\abstract{
        The \emph{tensor t-function}, a formalism that generalizes the well-known concept of matrix functions to third-order tensors, is introduced in [K.\ Lund, The tensor t-function: a definition for functions of third-order tensors, Numer.\ Linear Algebra Appl.\ 27 (3), e2288]. In this work, we investigate properties of the Fr\'echet derivative of the tensor t-function and derive algorithms for its efficient numerical computation.  Applications in condition number estimation and nuclear norm minimization are explored.  Numerical experiments implemented by the \texttt{t-Frechet} toolbox hosted at \url{https://gitlab.com/katlund/t-frechet} illustrate properties of the t-function Fr\'{e}chet derivative, as well as the efficiency and accuracy of the proposed algorithms.
}

\novelty{The Fr\'{e}chet derivative of the tensor t-function is defined, analyzed, and computed via multiple strategies with different performance profiles.}

\maketitle

  
\section{Introduction} \label{sec:intro}
Functions of matrices play an important role in many areas of applied mathematics and scientific computing, e.g., in network analysis~\cite{EstH10}, exponential integrators~\cite{HocO10}, physical simulations~\cite{Neu98} and statistical sampling~\cite{IliTS10}. This concept was generalized to functions of third-order tensors in~\cite{Lun20}, based on the tensor t-product formalism~\cite{Bra10,KilBHetal13,KilM11}; see also~\cite{MiaQW20} for a further extension to so-called \emph{generalized tensor functions}, which are functions of tensors with non-square faces. Functions (and generalized functions) of tensors have applications in deblurring of color images~\cite{ReiU22}, tensor neural networks~\cite{MalUHetal19,NewHAetal18}, multilinear dynamical systems~\cite{HooCB21}, and the computation of the tensor nuclear norm~\cite{BenEJetal22a}.

For functions of matrices, the Fr\'echet derivative is a well-established object with applications in, e.g.,  condition number estimation~\cite{Al-H09}, analysis of complex networks~\cite{DeJNetal22,Sch23b}, and the solution of matrix optimization problems~\cite{ThaDKetal17}. In this work, we consider the Fr\'echet derivative of functions of tensors, in order to generalize the above techniques to the tensor setting.

In addition to condition number estimation, the tensor Fr\'echet derivative has a number of potential applications, most notably in gradient descent procedures for nuclear norm minimization \cite{BenEJetal22,HosOM16,KreSS13,LiuZT19,LuFCetal20,LuPW19,YuaZ16,ZhaN19}.  Thanks to close connections with bivariate functions (see, e.g., \cite{Kre19} for the matrix function case), computational approaches for the tensor Fr\'echet derivative are a stepping stone towards solutions of tensor Lyapunov and Sylvester equations \cite{LiuJ21}.  Furthermore, a generalization of the network sensitivity measures discussed in~\cite{DeJNetal22,Sch23b} to multilayer networks (which can be represented as tensors) will also require a tensor Fr\'echet derivative.

This paper is organized as follows. In section~\ref{sec:foundations}, we collect several important definitions and results concerning matrix functions, the Fr\'echet derivative, and the tensor t-product.  Section~\ref{sec:tfunction} summarizes key results on the tensor t-function and introduces definitions and properties of its Fr\'echet derivative $L_f(\AA,\CC)$, including explicit Kronecker forms.  In Section~\ref{sec:algorithms} we discuss a number of methods for computing $L_f(\AA,\CC)$, drawing on well understood techniques such as Krylov subspace methods for matrix functions and fast Fourier transforms.  We examine applications such as the condition number of t-functions and the gradient of the tensor nuclear norm in Section~\ref{sec:applications}.  Finally, in Section~\ref{sec:experiments} we compare the performance of different algorithms for small- and medium-scale problems, and we summarize our findings in Section~\ref{sec:conclusion}.

\section{Foundations} \label{sec:foundations}
We recall important concepts from matrix function theory, Fr\'echet derivatives, and the t-product formalism that form the basis of this work.

\subsection{Functions of matrices} \label{sec:matfun}
Functions of matrices can be defined in many different ways, the three most popular of which are based on the Jordan canonical form, Hermite interpolation polynomials, and the Cauchy integral formula; see~\cite[Section~1.2]{Hig08b} for a thorough treatment. We recall two of the definitions that are particularly important for our work.

Let $A \in \spCnn$ be a matrix with spectrum $\spec(A) := \{\lambda_j\}_{j=1,\ldots,N}$, where $N \leq n$ and the $\lambda_j$ are distinct. Suppose that $A$ has Jordan canonical form,
\begin{equation} \label{eq:jordan}
A = X J X^\inv = X^\inv \diag(J_{m_1}(\lambda_{j_1}), \ldots, J_{m_p}(\lambda_{j_\ell})) X,
\end{equation}
where $J_m(\lambda_j)$ is an $m \times m$ Jordan block for an eigenvalue $\lambda_j$. Denote by $n_j$ the {\em index} of $\lambda_j$, i.e., the size of the largest Jordan block associated to~$\lambda_j$. (Note that eigenvalues may be repeated in the sequence $\{\lambda_{j_k}\}_{k=1}^\ell$). We then say that a function is {\em defined on the spectrum of $A$} if all the values $f^{(k)}(\lambda_j)$ for $k = 0, \ldots, n_j-1$ and $j = 1, \ldots, N$ exist.

If $f$ is defined on the spectrum of $A$ with Jordan form~\eqref{eq:jordan}, then we can define $f(A)$ via
\[
	f(A) := X f(J) X^\inv,
\]
where $f(J) := \diag(f(J_{m_1}(\lambda_{j_1})), \ldots, f(J_{m_p}(\lambda_{j_\ell})))$, and
\[
	f(J_{m_i}(\lambda_{j_k}))
	:=
	\begin{bmatrix}
    	f(\lambda_{j_k}) & f'(\lambda_{j_k}) & \frac{f''(\lambda_{j_k})}{2!}& \dots		& \frac{f^{(n_{j_k}-1)}(\lambda_{j_k})}{(n_{j_k}-1)!}	\\
    	0			 & f(\lambda_{j_k})	 & f'(\lambda_{j_k}) 			& \dots		& \vdots	\\
    	\vdots		 & \ddots		 & \ddots			  		& \ddots	& \frac{f''(\lambda_{j_k})}{2!}	\\
    	\vdots		 &				 & \ddots 			  		& \ddots    & f'(\lambda_{j_k})		\\
    	0			 & \dots		 & \dots 			  		& 0			& f(\lambda_{j_k})
	\end{bmatrix} \in \spC^{m_i \times m_i}.
\]
When $A$ is diagonalizable with $\spec(A) = \{\lambda_j\}_{j=1,\ldots,n}$ (possibly no longer distinct) the Jordan form definition greatly simplifies to
\[
f(A) = X \diag(f(\lambda_1), \ldots, f(\lambda_n)) X^\inv,
\]
where $\diag$ is the operator that maps an $n$-vector to its corresponding $n \times n$ diagonal matrix.

When $f$ is analytic on a region that contains $\spec(A)$, we can alternatively define $f(A)$ via the Cauchy integral formula,
\[
    f(A) := \frac{1}{2\pi i}\int_{\Gamma} f(\zeta)(\zeta I - A)^\inv \d\zeta,
\]
where $\Gamma$ is a path that winds around $\spec(A)$ exactly once.

When $f$ is analytic, so that both of the above definitions can be applied, the two definitions are equivalent and yield the same result; see~\cite[Theorem~1.12]{Hig08b}.

\subsection{The Fr\'echet derivative} \label{sec:frechet}
In the most general case, the Fr\'echet derivative is defined for functions between normed vector spaces $V, W$ (with respective norms $\norm{\cdot}_V, \norm{\cdot}_W$). Let $U \subset V$ be an open subset and let $f: U \longrightarrow W$. Then $f$ is \emph{Fr\'echet-differentiable} at $\vu \in U$ if there exists a bounded linear operator $L(\vu): V \rightarrow W$ such that
\begin{equation} \label{eq:frechet_general}
    \lim\limits_{\norm{\vh}_V \rightarrow 0} \frac{\norm{f(\vu+\vh)-f(\vu) - L(\vu)\vh}_W}{\norm{\vh}_V} = 0.
\end{equation}

When $f: \spCnn \longrightarrow \spCnn$ is a function of a matrix, one usually denotes the Fr\'echet derivative of $f$ at the matrix $A$ as $L_f(A, \cdot)$ (see, e.g.,~\cite[Chapter~3]{Hig08b}) and rephrases the condition~\eqref{eq:frechet_general} using the matrix two-norm and Landau notation as
\begin{equation} \label{eq:frechet_def}
    f(A + E) - f(A) = L_f(A,E) + o(\norm{{E}}), \quad \textnormal{for all } E \in \spCnn,
\end{equation}
for an appropriate matrix norm $\norm{\cdot}$.  A sufficient condition for $L_f(A,\cdot)$ to exist is that $f$ is $2n-1$ times continuously differentiable on a region containing $\spec(A)$ (see~\cite[Theorem~3.8]{Hig08b}).  If the Fr\'echet derivative exists, it is unique. 

In particular, the Fr\'echet derivative of a matrix function is guaranteed to exist if $f$ is analytic on a region containing $\spec(A)$, and in this case $L_f(A,E)$ has the integral representation
\begin{equation} \label{eq:frechet_derivative_integral}
    L_f(A,E) = \frac{1}{2\pi i} \int_\Gamma f(\zeta)(\zeta I - A)^{-1}E(\zeta I - A)^{-1}\d \zeta,
\end{equation}
where $\Gamma$ is again a path that winds around $\spec(A)$ exactly once; see, e.g.,~\cite{Hig08b,KanR17}. In addition to being of theoretical interest, the integral representation also forms the basis of efficient computational methods for approximating $L_f(A,E)$, in particular when $E$ is of low rank; see~\cite{KanKRetal21,KanR17,Kre19}, as well as~\cite{Sch23a} for an extension to higher-order Fr\'echet derivatives.

Related is the \emph{G\^ateaux (or directional) derivative} of $f$ at $A$, defined as
\begin{equation*} 
    G_f(A,E) = \lim\limits_{t\rightarrow 0} \frac{f(A+tE)-f(A)}{t}.
\end{equation*}
If $f$ is Fr\'echet-differentiable at $A$, all its directional derivatives exist and we have $G_f(A,E) = L_f(A,E)$ for all $E \in \spC^{n \times n}$. The converse is not necessarily true: even when all directional derivatives of $f$ at $A$ exist, $f$ need not be Fr\'echet-differentiable at $A$. 

\subsection{Tensors and the t-product} \label{sec:tensors}
In the context of this work, a \emph{tensor} is viewed as a multidimensional array, i.e., a generalization of the concept of vectors and matrices to higher dimensions. We restrict ourselves to third-order tensors, i.e., arrays in~$\spC^{n \times m \times p}$, as the t-product introduced in~\cite{Bra10, KilBHetal13, KilM11} is only defined in this case. Figure~\ref{fig:tensors} depicts the different ``views" of a third-order tensor, which are useful for visualizing the forthcoming concepts. We define the (Frobenius) norm of a tensor $\AA \in \spC^{n \times m \times p}$, with $\AA(i,j,k)$ denoting the $ijk$th entry, as
\begin{equation} \label{eq:tensor_norm}
    \norm{\AA}_F = \sqrt{\sum_{i=1}^{n}\sum_{j=1}^{m}\sum_{k=1}^{p} \lvert\AA(i,j,k)\rvert^2} \,,
\end{equation}
which can be seen as an analogue of the matrix Frobenius norm $\norm{\cdot}_F$.

\begin{figure}
\centering
    \begin{tabular}{cccccc}
		    \scalebox{.9}{\begin{tikzpicture}   
            \tikzcuboid{
                dimx = 5, dimy = 10, dimz = 1,
                scalex = .14, scaley = .15, scalez = .75
                };
        \end{tikzpicture}}%
        &
        \scalebox{.9}{\begin{tikzpicture}   
            \tikzcuboid{
                dimx = 5, dimy = 1, dimz = 5,
                scalex = .15, scaley = 1.5, scalez = .15
                };
        \end{tikzpicture}}%
        &
        \scalebox{.9}{\begin{tikzpicture}   
            \tikzcuboid{
                dimx = 1, dimy = 10, dimz = 5,
                scalex = .75, scaley = .15, scalez = .15
                };
        \end{tikzpicture}}%
        &
        \scalebox{.9}{\begin{tikzpicture}   
            \tikzcuboid{
                dimx = 1, dimy = 1, dimz = 5,
                scalex = .75, scaley = 1.5, scalez = .15
                };
        \end{tikzpicture}}%
        &
        \scalebox{.9}{\begin{tikzpicture}   
            \tikzcuboid{
                dimx = 1, dimy = 10, dimz = 1,
                scalex = .75, scaley = .15, scalez = .75
                };
        \end{tikzpicture}}%
        &
        \scalebox{.9}{\begin{tikzpicture}   
            \tikzcuboid{
                dimx = 5, dimy = 1, dimz = 1,
                scalex = .15, scaley = 1.5, scalez = .75
                };
        \end{tikzpicture}}%
				\\
        \small (a) & \small (b) & \small (c) & \small (d) & \small (e) & \small (f)
    \end{tabular}
    \caption{Different views of a third-order tensor $\AA \in \spCtensor$. (a) tube fibers: $\AA(:,j,k)$; (b) column fibers: $\AA(i,:,k)$; (c) row fibers: $\AA(i,j,:)$; (d) frontal slices: $\AA(i,:,:)$; (e) lateral slices: $\AA(:,j,:)$; (f) horizontal slices: $\AA(:,:,k)$ \label{fig:tensors}}
\end{figure}

As the t-product formalism makes extensive use of block matrices, we introduce basic notations for these. Define the standard block unit vectors $\vE_k^{np \times n} := \ve_k^p \otimes I_n$, where $\ve_k^p \in \spC^p$ is the $k$th canonical unit vector in $\spC^p$, and $I_n$ is the $n \times n$ identity matrix. When the dimensions are clear from context, we drop the sub- or superscripts. 

The tensor t-product~\cite{Bra10, KilBHetal13, KilM11} defines a way to multiply third-order tensors, based on viewing them as stacks of frontal slices (as in Figure~\ref{fig:tensors}(d)).  Let $\AA \in\spC^{n \times m \times p}, \BB \in \spC^{m \times s \times p}$ and denote their frontal faces, respectively, as $A^{(k)}$ and $B^{(k)}$, $k = 1, \ldots, p$. The operations $\texttt{unfold}$ and $\texttt{fold}$ transform the tensor $\AA$ into a block vector of size $np \times m$ and vice versa, i.e., 
\[
    \unfold{\AA} :=
    \begin{bmatrix}
        A^{(1)}	\\
        A^{(2)}	\\
        \vdots	\\
        A^{(p)}
    \end{bmatrix}, \mbox{ and } \fold{\unfold{\AA}} := \AA.
\]
Additionally, \bcirccode turns $\AA$ into a block-circulant matrix of size $np \times mp$,
\begin{equation*}
    \bcirc{\AA} :=
    \begin{bmatrix}
        A^{(1)}	& A^{(p)}	& A^{(p-1)}		& \cdots	& A^{(2)}	\\
        A^{(2)}	& A^{(1)}	& A^{(p)}		& \cdots	& A^{(3)}	\\
        \vdots	& \ddots	& \ddots		& \ddots	& \vdots	\\
        A^{(p)}	& A^{(p-1)}	& \cdots		& A^{(2)}	& A^{(1)}
    \end{bmatrix}.
\end{equation*}
Note that the operators \texttt{fold}, \texttt{unfold}, and \bcirccode are linear.  As a shorthand, we use the term \emph{$n$-block circulant matrix} for a block circulant matrix with $n \times n$ blocks.

Using the above operators, the {\em t-product} of the tensors $\AA$ and $\BB$ is given as
\[
    \AA * \BB := \fold{\bcirc{\AA} \unfold{\BB}}.
\]
Many important concepts well-known for matrices, such as an identity element, inverses, transposition, and eigendecomposition, can also be defined for third-order tensors within the t-product framework; see \cite{Bra10, KilM11, KilBHetal13}.

Transposition of tensors is defined face-wise, i.e., $\AA^H$ is the $m \times n \times p$ tensor obtained by taking the conjugate transpose of each frontal slice of $\AA$ and then reversing the order of the second through $p$th transposed slices. For tensors with $n \times n$ square faces, there is an identity tensor $\II_{n \times n \times p} \in \spC^{n \times n \times p}$, whose first frontal slice is the $n \times n$ identity matrix $I_n$ and whose remaining frontal slices are all zero, which fulfills
\[
    \AA * \II_{n \times n \times p} = \AA = \II_{n \times n \times p} * \AA.
\]
We drop the subscript on $\II$ when the dimensions are clear from context.

When $n = m$, a unique inverse tensor $\AA^\inv$ can be defined as expected: if there exists $\BB \in \spC^{n \times n \times p}$ such that
\begin{equation} \label{eq:t-inv}
    \BB * \AA = \II = \AA * \BB,
\end{equation}
then $\AA^\inv := \BB$.

If $\AA \in \spC^{n \times n \times p}$ has diagonalizable faces, i.e., $A^{(k)} = X^{(k)} D^{(k)} \left(X^{(k)}\right)^\inv$, for all $k = 1, \ldots, p$, a tensor eigendecomposition can be defined via
\begin{equation} \label{eq:t-eigendecomp}
    \AA = \XX * \DD * \XX^\inv \mbox{ and } \AA * \XXarrow_i = \XXarrow_i * \vd_i,
\end{equation}
where $\XX$ and $\DD$ are the tensors whose faces are $X^{(k)}$ and $D^{(k)}$, respectively; $\XXarrow_i$ are the $n \times 1 \times p$ lateral slices of $\XX$ (see Figure~\ref{fig:tensors}(e)); and $\vd_j$ are the $1 \times 1 \times p$ tube fibers of $\DD$ (see Figure~\ref{fig:tensors}(a)).

\subsection{Block circulant matrices and the discrete Fourier transform}
It is well established that the discrete Fourier transform (DFT) unitarily diagonalizes circulant matrices \cite{Dav12}, and in \cite{KilBHetal13, KilM11} a block version of this result is shown to hold.  Namely, letting $F_p$ denote the $p \times p$ DFT and $\otimes$ the Kronecker product, it follows for $\AA \in \spC^{n \times n \times p}$ that
\begin{equation} \label{eq:dft}
    (F_p \otimes I_n) \bcirc{\AA} (F_p^H \otimes I_n) = \blkdiag(D_1, \ldots, D_p),
\end{equation}
where each $D_i, i = 1, \ldots, p$ is an $n \times n$ matrix, and $\blkdiag$ works similarly to $\diag$, but instead places matrices on the diagonal.

Another useful tool when working with block circulant matrices is the block circulant shift operator,
\begin{equation}\label{eq:shift_matrix}
    S_{n,p} := \begin{bmatrix}
        & & & I_n \\
        I_n & & & \\
        & \ddots & & \\
        & & I_n &
    \end{bmatrix} \in \spR^{np \times np},
\end{equation}
which is clearly unitary.  Using $S_{n,p}$, define the transformation
\begin{equation}\label{eq:circulant_shift}
    \SS_{n,p}: M \mapsto S_{n,p} M S_{n,p}^T.
\end{equation}
A matrix $M \in \spC^{np \times np}$ is block circulant if and only if $\SS_{n,p}(M) = M$. In the following sections, when dimensions and block sizes are clear from the context, we omit the corresponding indices and just write $S$ and $\SS$.

\section{The tensor t-function} \label{sec:tfunction}
In~\cite{Lun20}, a definition for functions of third-order tensors based on the t-product is given, generalizing the usual concept of matrix functions discussed in Section~\ref{sec:matfun}. Precisely, the action of the {\em tensor t-function} $f$ of $\AA \in \spC^{n \times n \times p}$ on another tensor $\BB \in \spC^{n \times s \times p}$ is defined as
\begin{equation} \label{eq:fAABB}
    \fAABB := \fold{f(\bcirc{\AA}) \cdot \unfold{\BB}}.
\end{equation}
By taking $\BB$ to be the identity tensor, $\BB = \II_{n \times n \times p}$, one obtains the t-function $f(\AA)$ via
\begin{equation} \label{eq:fAA}
    f(\AA) := \fold{f(\bcirc{\AA}) \cdot \unfold{\II_{n \times n \times p}}} = \fold{f(\bcirc{\AA}) \vE_1^{np \times n}}.
\end{equation}
Note in particular that when $f(z) = z^\inv$, we recover the definition of the tensor inverse \eqref{eq:t-inv}; see \cite[Theorem~5(iv)]{Lun20}.

The definitions~\eqref{eq:fAABB} and~\eqref{eq:fAA} boil down to evaluating the action of a matrix function (in the usual sense) on a block vector. The t-function therefore inherits many useful properties from matrix functions.

\begin{theorem}[Theorem~6 in~\cite{Lun20}] \label{thm:tensor_t-func_props}
	Let $\AA \in \spC^{n \times n \times p}$, and let $f: \spC \to \spC$ be defined on a region in the complex plane containing the spectrum of $\bcirc{\AA}$.  For part~(iv), assume that $\AA$ has an eigendecomposition as in equation~\eqref{eq:t-eigendecomp}, with $\AA * \XXarrow_i = \DD * \XXarrow_i = \XXarrow_i * \vd_i$, $i = 1, \ldots, n$. Then it holds that
	\begin{enumerate}[(i)]
		\item $f(\AA)$ commutes with $\AA$; \label{tfunc:commutativity}
		\item $f(\AA^H) = f(\AA)^H$; \label{tfunc:transpose}
		\item $f(\XX * \AA * \XX^\inv) = \XX f(\AA) \XX^\inv$; and \label{tfunc:similarity}
		\item $f(\DD) * \XXarrow_i = \XXarrow_i * f(\vd_i)$, for all $i = 1, \ldots, n$. \label{tfunc:diagonal}
	\end{enumerate}
\end{theorem}

\subsection{The derivative of the tensor t-function} \label{sec:tfrechet}
In view of~\eqref{eq:fAA}, which defines the tensor t-function in terms of a matrix function of a block-circulant matrix, it appears natural to define its Fr\'echet derivative accordingly.

\begin{lemma} \label{lem:tfrechet}
    Let $\AA \in \spC^{n \times n \times p}$ and let $f$ be $2np-1$ times continuously differentiable on a region containing $\spec(\bcirc{\AA})$. Then the Fr\'echet derivative of $f$ at $\AA$ exists, and for any $\CC \in \spC^{n \times n \times p}$,
    \begin{equation} \label{eq:t-frechet}
        L_f(\AA,\CC) = \fold{L_f(\bcirc{\AA}, \bcirc{\CC}) \vE_1^{np \times n}}.
    \end{equation}
\end{lemma}
\begin{proof}
    The operator $L_f(\bcirc{\AA}, \cdot)$ is the Fr\'echet derivative of $f$ at a matrix of size $np \times np$, so its existence is guaranteed by~\cite[Theorem~3.8]{Hig08b} under the assumptions of the lemma. Now consider the difference 
    \begin{equation} \label{eq:tfrechet_diff}
        f(\AA + \CC) - f(\AA) = \fold{f(\bcirc{\AA + \CC})\vE_1^{np \times n}} - \fold{f(\bcirc{\AA})\vE_1^{np \times n}}
    \end{equation}
    
    Using linearity of \bcirccode, \texttt{fold}, and matrix multiplication, we can rewrite \eqref{eq:tfrechet_diff} as
    \begin{align}
        f(\AA + \CC) - f(\AA) &= \fold{f\left( \bcirc{\AA + \CC} \right)\vE_1^{np \times n} - f\left( \bcirc{\AA} \right) \vE_1^{np \times n}} \nonumber\\
                              &= \fold{\left( f \left(\bcirc{\AA + \CC} \right) - f\left( \bcirc{\AA} \right) \right) \vE_1^{np \times n}} \nonumber\\
                              &= \fold{ \left( f \left( \bcirc{\AA}+\bcirc{\CC} \right) - f\left( \bcirc{\AA} \right) \right) \vE_1^{np \times n}} \nonumber\\
                              &= \fold{\left(L_f\left( \bcirc{\AA},\bcirc{\CC} \right)
                              +o\left( \norm{\bcirc{\CC}}_F \right) \right) \vE_1^{np \times n}}\nonumber\\
                              &= \fold{\left( L_f\left( \bcirc{\AA},\bcirc{\CC} \right) \right) \vE_1^{np \times n}}
                              + o\left(\norm{\bcirc{\CC}}_F \right),\label{eq:t-frechet_proof}
    \end{align}
    where we have used definition~\eqref{eq:frechet_def} in the second-to-last equality. 
    
    Due to the special structure of $\bcirc{\CC}$, each of its $np \times n$ block-columns fulfills
    \[
    \norm{[\bcirc{\CC}]_{:,(i-1)\cdot n:i\cdot n}}_F = \norm{\CC}_F, i = 1,\ldots,p,
    \]
    so that in total $\norm{\bcirc{\CC}}_F = \sqrt{p}\norm{\CC}$. Therefore, $o(\norm{\bcirc{\CC}}_F) = o(\norm{\CC}_F)$ and it follows from~\eqref{eq:t-frechet_proof} that~\eqref{eq:t-frechet} is indeed the Fr\'echet derivative of $f(\AA)$ in the sense of definition~\eqref{eq:frechet_general}. 
\end{proof}

If the assumptions of Lemma~\ref{lem:tfrechet} are fulfilled, we also say that $f$ is \textit{t-Fr\'echet differentiable} at $\AA$.

A similar relation holds for the G\^ateaux derivative.
\begin{proposition}
    Let $f$ be G\^ateaux-differentiable at $\bcirc{\AA}$. Then $f$ is G\^ateaux-differentiable at $\AA$, and
    \begin{equation} \label{eq:t-gateaux}
        G_f(\AA,\CC) = \fold{G_f(\bcirc{\AA}, \bcirc{\CC}) \vE_1^{np \times n}}.
    \end{equation}
\end{proposition}
\begin{proof}
    The proof follows directly from the definition of the G\^ateaux derivative, by inserting the  definition~\eqref{eq:fAA} of the tensor t-function and again exploiting the linearity of \texttt{fold} and \bcirccode. Consequently, we find
    \begin{align*}
        G_f(\AA,\CC) & = \lim\limits_{t\rightarrow 0} \frac{f(\AA + t\CC)-f(\AA)}{t} \\
        &=  \lim\limits_{t\rightarrow 0} \frac{\fold{f(\bcirc{\AA+ t\CC}) \vE_1^{np \times n}}-\fold{f(\bcirc{\AA}) \vE_1^{np \times n}}}{t}\\
        &=  \lim\limits_{t\rightarrow 0} \frac{\fold{f(\bcirc{\AA+ t\CC})\vE_1^{np \times n}-f(\bcirc{\AA}) \vE_1^{np \times n}}}{t}\\
        &=  \lim\limits_{t\rightarrow 0} \frac{\fold{(f(\bcirc{\AA+ t\CC})-f(\bcirc{\AA})) \vE_1^{np \times n}}}{t}\nonumber\\
        &=  \lim\limits_{t\rightarrow 0} \frac{\fold{(f(\bcirc{\AA} + t\cdot\bcirc{\CC})-f(\bcirc{\AA})) \vE_1^{np \times n}}}{t}\\
        &=  \fold{\lim\limits_{t\rightarrow 0} \frac{(f(\bcirc{\AA} + t\cdot\bcirc{\CC})-f(\bcirc{\AA}))}{t} \vE_1^{np \times n}}\\
        &=  \fold{G_f(\bcirc{\AA}, \bcirc{\CC}) \vE_1^{np \times n}},
    \end{align*}
    which is exactly~\eqref{eq:t-gateaux}.
\end{proof}

\begin{remark} \label{rem:frechet_gateaux}
As in the matrix case, when $f$ is Fr\'echet-differentiable at $\AA$, then its Fr\'echet and G\^ateaux derivative coincide:
\[
L_f(\AA, \CC) = G_f(\AA, \CC).
\]
\end{remark}

\begin{remark} \label{rem:bcirc}
In the derivation of the G\^ateaux derivative, one can observe that when $A,C \in \spC^{np \times np}$ are both $n$-block circulant matrices, then $L_f(A,C) = G_f(A,C)$ is also $n$-block circulant.
\end{remark}

\subsection{Properties of the t-Fr\'echet derivative}\label{sec:properties}
As it is defined in terms of the Fr\'echet derivative of a matrix function, the t-Fr\'echet derivative~\eqref{eq:t-frechet} also inherits many of the properties of the matrix function derivative, which we collect in the following lemma.

\begin{lemma}\label{lem:properties_tfrechet}
Let $\AA \in \spC^{n \times n \times p}$ and let $g_1$ and $g_2$ be t-Fr\'echet differentiable at $\AA$. Then
    \begin{enumerate}[(i)]
        \item $f_1 = \alpha g_1 + \beta g_2$ is t-Fr\'echet differentiable at $\AA$, and \label{item:sumrule}
        \[
        L_{f_1}(\AA,\CC) = \alpha L_{g_1}(\AA,\CC) + \beta L_{g_2}(\AA,\CC).
        \]
        \item $f_2 = g_1g_2$ is t-Fr\'echet differentiable at $\AA$, and \label{item:productrule}
        \[
        L_{f_2}(\AA,\CC) = L_{g_1}(\AA,\CC)g_2(\AA) + g_1(\AA)L_{g_2}(\AA,\CC).
        \]
        \item If further $h$ is t-Fr\'echet differentiable at $h(\AA)$, then $f_3 = h \circ g_1$ is t-Fr\'echet differentiable at $\AA$, and \label{item:chainrule}
        \[
        L_{f_3}(\AA,\CC) = L_h(g_1(\AA),L_{g_1}(\AA,\CC)).
        \]
    \end{enumerate}
\end{lemma}
\begin{proof}
    Let $A, C$ denote $\bcirc{\AA}, \bcirc{\CC}$, respectively.  For part~(\ref{item:sumrule}), observe that by~\eqref{eq:t-frechet}, we have
    \begin{align*}
        L_{f_1}(\AA,\CC) &= \fold{L_{f_1}(A,\CC}\vE_1^{np \times n}) \\
                         &= \fold{\Big(\alpha L_{g_1}(A,\CC)
                         + \beta L_{g_2}(A,\CC)\Big)\vE_1^{np \times n}} \\
                         &= \alpha \cdot \fold{L_{g_1}(A,\CC}\vE_1^{np \times n})
                         + \beta \cdot \fold{L_{g_2}(A,\CC}\vE_1^{np \times n}) \\
                         &= \alpha L_{g_1}(\AA,\CC) + \beta L_{g_2}(\AA,\CC),
    \end{align*}
    where the second equality follows from~\cite[Theorem~3.2]{Hig08b} and the third equality follows from the linearity of \texttt{fold}. In a completely analogous fashion, part~(\ref{item:productrule}) and~(\ref{item:chainrule}) follow from their respective matrix function counterparts~\cite[Theorem~3.3 \& Theorem~3.4]{Hig08b}.
\end{proof}

We also have an analogous relation to the integral representation~\eqref{eq:frechet_derivative_integral}.

\begin{lemma}\label{lem:integral_representation}
    Let $f$ be analytic on a region containing $\spec(\bcirc{\AA})$. Then
    \begin{equation*} 
        L_f(\AA,\CC) = \frac{1}{2\pi i} \int_\Gamma f(\zeta)(\zeta \II - \AA)^{-1} * \CC * (\zeta \II - \AA)^{-1}\d \zeta,
    \end{equation*}
    where the inverse is defined as in \eqref{eq:t-inv}.
\end{lemma}
\begin{proof}
    Let $A_\zeta, C$ denote $\bcirc{\zeta \II - \AA}, \bcirc{\CC}$, respectively. By~\eqref{eq:frechet_derivative_integral} applied to $L_f(\AA,\CC)$ and the linearity of \texttt{fold}, it follows that
    \begin{equation} \label{eq:proof_integral1}
        L_f(\AA,\CC)
        = \frac{1}{2\pi i} \int_\Gamma f(\zeta) \fold{A_\zeta^{-1} C A_\zeta^{-1}\vE_1^{np \times n}} \d \zeta.
    \end{equation}
    Noting that $A_\zeta^{-1}\vE_1^{np \times n} = \unfold{(\zeta \II - \AA)^{-1}}$, we have
    \[
        C A_\zeta^{-1}\vE_1^{np \times n} = \unfold{\CC * (\zeta \II - \AA)^{-1}},
    \]
    so that~\eqref{eq:proof_integral1} becomes
    \begin{align*}
        L_f(\AA,\CC) &= \frac{1}{2\pi i} \int_\Gamma f(\zeta) \fold{A_\zeta^{-1}\unfold{\CC * (\zeta \II - \AA)^{-1}}}\d \zeta\\
                     &= \frac{1}{2\pi i} \int_\Gamma f(\zeta) (\zeta \II - \AA)^{-1} * \CC * (\zeta \II - \AA)^{-1}\d \zeta.
    \end{align*}
\end{proof}

\subsection{Explicit representation of the t-Fr\'echet derivative}
An intuitive way to compute $L_f(\AA, \CC)$  for a particular direction tensor $\CC$ is based on a well known relation for the matrix Fr\'echet derivative. For matrices $A, C \in \spC^{np \times np}$, if $f$ is $2np-1$ times continuously differentiable on a region containing $\spec(A)$, we have
\begin{equation} \label{eq:frechet_2x2}
    f\left(
    \begin{bmatrix}
      A & C \\ O_{np \times np} & A 
    \end{bmatrix}\right)
    = 
    \begin{bmatrix}
      f(A) & L_f(A,C) \\ O_{np \times np} & f(A)
    \end{bmatrix},
\end{equation}
where $O_{np \times np}$ denotes an $np \times np$ matrix of zeros; see~\cite[eq.~(3.16)]{Hig08b}. Thus, $L_f(A,C)$ can be found by first evaluating $f$ at a $2np \times 2np$ block upper triangular matrix and then extracting the top-right block, 
\begin{equation} \label{eq:frechet_2x2_2}
    L_f(A,C)
    =
    \begin{bmatrix}
       I_{np} & O_{np \times np}
    \end{bmatrix}
    \cdot f\left(
    \begin{bmatrix}
        A & C \\ O_{np \times np} & A
    \end{bmatrix}
    \right)
    \cdot
    \begin{bmatrix}
        O_{np \times np} \\ I_{np}
    \end{bmatrix}.
\end{equation}
In the context of the Fr\'echet derivative of the t-function,~\eqref{eq:frechet_2x2_2} turns into
\begin{equation*}
    L_f(\AA,\CC)
    = \fold{
        \begin{bmatrix}
            I_{np} & O_{np \times np}
        \end{bmatrix}
        \cdot
        f\left(
        \begin{bmatrix}
            A & C \\ O_{np \times np} & A
        \end{bmatrix}
        \right)
        \cdot
        \begin{bmatrix}
            O_{np \times n} \\ I_n \\ O_{n(p-1) \times n}
        \end{bmatrix}
    },
\end{equation*}
where $A = \bcirc{\AA}$, $C = \bcirc{\CC}$, and we have used the fact that
\[
    \begin{bmatrix}
        O_{np \times np} \\ I_{np}
    \end{bmatrix}
    \vE_1^{np \times n}
    =
    \begin{bmatrix}
        O_{np \times np} \\ I_{np}
    \end{bmatrix}
    \cdot
    \begin{bmatrix}
        I_n \\ O_{n(p-1) \times n}
    \end{bmatrix}
    =
    \begin{bmatrix}
        O_{np \times n} \\ I_n \\ O_{n(p-1) \times n}
    \end{bmatrix}.
\]
We can thus explicitly write the Fr\'echet derivative of the t-function $f(\AA)$ in the direction $\CC$ in terms of the product of a matrix function acting on a block vector, wherein the upper half of the resulting block vector is extracted and folded back into a tensor.  In summary,
\begin{equation} \label{eq:tfrechet_compact}
    L_f(\AA,\CC) = \fold{
    \left[
        f\left(
            \begin{bmatrix}
                \bcirc{\AA} & \bcirc{\CC} \\ O_{np \times np} & \bcirc{\AA}
            \end{bmatrix}
        \right)
        \begin{bmatrix}
            O_{np \times n} \\ I_n \\ O_{n(p-1) \times n}
        \end{bmatrix}
    \right]_{1:np,:}
    }.
\end{equation}

\subsection{Kronecker forms of the t-Fr\'echet derivative} \label{sec:kronecker}
The Fr\'echet derivative induces a linear mapping $L_f(\AA, \cdot): \spC^{n \times n \times p} \longrightarrow \spC^{n \times n \times p}$. Thus, identifying $\spC^{n \times n \times p}$ with $\spC^{n^2p}$, there is a matrix representation $K_f(\AA) \in \spC^{n^2p \times n^2p}$ such that for any $\CC \in \spC^{n \times n \times p}$
\begin{equation} \label{eq:kronecker_form_def}
    \vec{L_f(\AA,\CC)} = K_f(\AA)\vec{\CC},
\end{equation}
where $\vec{\cdot}$ stacks the entries of a tensor into a column vector. The matrix $K_f(\AA)$ is also called the \emph{Kronecker form} of the Fr\'echet derivative. (See, e.g,~\cite[Section~3.2]{Hig08b} for the matrix function case.)

For computing the Kronecker form, one can simply evaluate the Fr\'echet derivative $L_f(\AA, \cdot)$ on all tensors of the canonical basis $\{\EE_{ijk} : i,j = 1,\ldots,n, k=1,\ldots,p\}$ of $\spC^{n \times n \times p}$ (i.e., $\EE_{ijk}$ is a tensor with entry one at position $(i,j,k)$ and all other entries zero). We summarize this discussion in the following definition. 
\begin{definition}
    Let $f$ be t-Fr\'echet differentiable at $\AA \in \spC^{n \times n \times p}$. The \emph{Kronecker form} of $L_f(\AA, \cdot)$ is the matrix $K_f(\AA) \in \spC^{n^2p \times n^2p}$ with columns $\vk_\ell, \ell = 1,\ldots,n^2p$ defined via
    \begin{equation}\label{eq:kronecker_form_columns}
        \vk_{i+(k-1)n+(j-1)np} = \vec{L_f(\AA,\EE_{ijk})}.
    \end{equation}
\end{definition}

A simple computational procedure for forming the Kronecker form is outlined in Algorithm~\ref{alg:full_kronecker}, where we use MATLAB-style colon notation, i.e., $a:b$ means all indices between (and including) $a$ and $b$.

\begin{algorithm}[t]
\caption{\label{alg:full_kronecker} Kronecker form of the t-Fr\'echet derivative}
    \begin{algorithmic}[1]
    \setstretch{1.2}
    \smallskip
    
    \Statex \textbf{Input:}\ \ \ $f$, $\AA$
    \Statex \textbf{Output:} Kronecker form $K = K_f(\AA)$
    \smallskip
    
    \For{$i = 1,\ldots,n$}
        \For{$j = 1,\ldots,n$}
            \For{$k = 1,\ldots,p$}
                \State $\YY \leftarrow L_f(\AA, \EE_{ijk})$
                \State $K(1:n^2p,i+(k-1)n+(j-1)np) \leftarrow \vec{\YY}$
            \EndFor
        \EndFor
    \EndFor
    \end{algorithmic}
\end{algorithm}

\begin{remark} \label{rem:computational_cost_kronecker_form}
    We note that the computational cost of Algorithm~\ref{alg:full_kronecker} is extremely high, making it infeasible even for medium scale problems (a situation that is similar already for matrix functions): computing a single Fr\'echet derivative $L_f(\AA,\EE_{ijk})$ using the relation~\eqref{eq:tfrechet_compact} and a dense matrix function algorithm for evaluating $f$ has a cost of $\mathcal{O}(n^3p^3)$ for most practically relevant functions $f$. Then, forming $\KK_f(\AA)$ via Algorithm~\ref{alg:full_kronecker} costs $\mathcal{O}(n^5p^4)$ flops and requires $\mathcal{O}(n^4p^2)$ storage. Thus, the Kronecker form can typically not be used in actual computations, but it is a useful theoretical tool, e.g., for defining condition numbers; see Section~\ref{sec:condition}.
\end{remark}

The tensor t-function is intimately related to matrix functions of block-circulant matrices. It is therefore interesting to examine the relationship between the Kronecker form $K_f(\AA)$ of the t-Fr\'echet derivative and the Kronecker form $K_f(\bcirc{\AA})$ of the Fr\'echet derivative of the matrix function $f(\bcirc{\AA})$. Note that $K_f(\bcirc{\AA}) \in \spC^{n^2p^2 \times n^2p^2}$, so that both matrices cannot coincide, but it turns out that they are still highly related. To make the connection precise, we first need the following auxiliary result.

\begin{proposition}\label{pro:E_ijk}
    Let $\EE_{ijk}$ be the unit tensor with a $1$ only in position $(i,j,k)$ and zeroes everywhere else. Then, with $E_{IJ} \in \spC^{np \times np}$ as the matrix that is zero everywhere except for a $1$ at $I = i + (k-1)n, J = j$,\footnote{In other words, $E_{IJ} = \ve_1^T \otimes \unfold{\EE_{ijk}}$, $\ve_1 \in \spC^p$ is the matrix that is zero everywhere except its first $np \times n$ block column, which is $\unfold{\EE_{i,j,k}}$.}
    \begin{equation*}
        \bcirc{\EE_{ijk}} = \sum_{\ell = 0}^{p-1} \SS^\ell(E_{IJ}).
    \end{equation*}
\end{proposition}
\begin{proof}
    The result immediately follows by noting that $(I,J)$ as defined above is one particular nonzero entry of $\bcirc{\EE_{ijk}}$, and, by the definition of $\SS$, the sequence of matrices $\SS^\ell(E_{IJ})$ cyclically moves through all other of its nonzero entries.\footnote{As $\SS^p(E_{IJ}) = E_{IJ}$, one could also start with $(I,J)$ corresponding to any other particular nonzero entry of $\bcirc{\EE_{ijk}}$, not necessarily the one given in the assertion.} 
\end{proof}

Due to the linearity of the Kronecker product, we thus have that
\begin{equation}\label{eq:relation_kronecker_Eijk}
L_f(\bcirc{\AA},\bcirc{\EE_{ijk}}) = \sum_{\ell = 0}^p L_f(\bcirc{\AA},\SS^\ell (E_{IJ})),
\end{equation}
with $E_{IJ}$ as defined in Proposition~\ref{pro:E_ijk}. The Fr\'echet derivatives on the right-hand side of~\eqref{eq:relation_kronecker_Eijk}, when vectorized, correspond to $p$ columns of the Kronecker form $K_f(\bcirc{\AA})$. Further, by~\eqref{eq:t-frechet} and~\eqref{eq:kronecker_form_columns}, the first $n^2p$ entries of the left-hand side of~\eqref{eq:relation_kronecker_Eijk} correspond to a column of $K_f(\AA)$. Thus, each column of $K_f(\AA)$ equals the sum of (the first $n^2p$ entries) of $p$ columns of $K_f(\bcirc{\AA})$, and each column of $K_f(\bcirc{\AA})$ appears in \emph{exactly one} of those sums.

The indices of the columns of $K_f\left(\bcirc{\AA} \right)$ that contribute to a particular column of $K_f(\AA)$ can be obtained by carefully inspecting how the index $(I,J)$ is moved around under the cyclical shifts $\SS^{\ell}$.
\begin{lemma}\label{lem:column_relation}
Let $\AA \in \spC^{n \times n \times p}$, let $f$ be analytic on a region containing the spectrum of $\bcirc{\AA}$, and let $K_1 := K_f\left(\AA \right)$ and $K_2 := K_f\left(\bcirc{A} \right)$ denote the Kronecker forms of the Fr\'echet derivatives of the t-function $f(\AA)$ and the matrix function $f(\bcirc{\AA})$, respectively. Then, for $c := i+(k-1)n+(j-1)np$, we have
\[
K_1(:, c) = \sum_{\alpha = 1}^p K_2(1:n^2,c + s_\alpha),
\]
where
\[
s_\alpha = \begin{cases}
0, & \text{if $\alpha = 1$} \\
s_{\alpha-1} + n^2p-np+1, & \text{if $\alpha = p-k+2$} \\
s_{\alpha-1} + n^2p+n,  & \text{otherwise.}
\end{cases}
\]
\end{lemma}
\begin{proof}
The result follows from Proposition~\ref{pro:E_ijk} by observing how $\SS$ acts on a unit matrix $E_{IJ}$. The application of $\SS$ cyclically shifts each block of the matrix one block column to the right and one block row down. Thus, as all blocks are $n \times n$, as long as the single nonzero entry of $E_{IJ}$ is not in the last block row or column, it is moved by exactly $n$ entries to the right and $n$ entries down, corresponding to $n^2p+n$ entries when vectorizing. Due to our choice of $E_{IJ}$ in Proposition~\ref{pro:E_ijk}, its nonzero entry lies in the $k$th block of the first block column. Therefore, this nonzero entry reaches the last block row after $p-k$ applications of $\SS$ and then moves to the first block row with the $p-k+1$st application. Thus, it moves $n$ positions to the right and $n(p-1)$ positions \emph{up}. This corresponds to $n^2p-np+1$ entries after vectorization.
\end{proof}

To verify that Lemma~\ref{lem:column_relation} is indeed true and to get a better handle on the rather unintuitive indexing scheme, the reader is encouraged to run and examine the script \texttt{test\_t\_func\_cond.m} in the \texttt{t-frechet} code repository described in Section~\ref{sec:experiments}.

A further interesting observation is obtained by viewing the relations we have derived so far ``in the opposite direction." It then turns out that it is sufficient to compute $n^2$ Fr\'echet derivatives in order to obtain all columns of the $n^2p^2 \times n^2p^2$ matrix $K_f(\bcirc{\AA})$ (and thus, in light of Lemma~\ref{lem:column_relation}, all columns of $K_f(\AA)$ as well). This is due to the following result.

\begin{proposition}\label{pro:shift_EIJ}
Let $\AA \in \spC^{n \times n \times p}$ and let $f$ be analytic on a region containing $\spec(\bcirc{\AA})$. Further, let $S$ denote the shift matrix defined in~\eqref{eq:shift_matrix} and let $E_{IJ} \in \spC^{n^2p^2 \times n^2p^2}$ be a matrix with $1$ only in position $(I,J)$ and $0$ everywhere else. Then, for any integers $\ell_1, \ell_2 \geq 0$,
\[
L_f(\bcirc{\AA}, S^{\ell_1}E_{IJ}(S^T)^{\ell_2}) = S^{\ell_1} \left(L_f(\bcirc{\AA}, E_{IJ})\right) (S^T)^{\ell_2}.
\]
\end{proposition}
\begin{proof}
By~\cite[Eq.~(3.24)]{Hig08b}, for any $C \in \spC^{np \times np}$ we have the relation
\begin{equation}\label{eq:power_series_frechet}
    L_f(\bcirc{\AA}, C) = \sum\limits_{\alpha=1}^\infty a_\alpha \sum\limits_{\beta=1}^\alpha \bcirc{\AA}^{\beta-1}C \bcirc{\AA}^{\alpha-\beta},
\end{equation}
using the power series representation $f(z) = \sum_{\alpha=0}^\infty a_\alpha z^\alpha$. Inserting $S^{\ell_1}E_{IJ}(S^T)^{\ell_2}$ instead of $C$ in relation~\eqref{eq:power_series_frechet}, we find that
\begin{align*}
    & \ \ \ L_f(\bcirc{\AA}, S^{\ell_1}E_{IJ}(S^T)^{\ell_2}))\\
    &= \sum\limits_{\alpha=1}^\infty a_\alpha \sum\limits_{\beta=1}^\alpha \bcirc{\AA}^{\beta-1}S^{\ell_1} E_{IJ}(S^T)^{\ell_2} \bcirc{\AA}^{\alpha-\beta} \\
    &= \sum\limits_{\alpha=1}^\infty a_\alpha \sum\limits_{\beta=1}^\alpha S^{\ell_1}\bcirc{\AA}^{\beta-1}(S^T)^{\ell_1}S^{\ell_1} E_{IJ}(S^T)^{\ell_2} S^{\ell_2}\bcirc{\AA}^{\alpha-\beta}(S^T)^{\ell_2} \\
    &= S^{\ell_1}\left(\sum\limits_{\alpha=1}^\infty a_\alpha \sum\limits_{\beta=1}^\alpha \bcirc{\AA}^{\beta-1}E_{IJ}\bcirc{\AA}^{\alpha-\beta} \right)(S^T)^{\ell_2} \\
    &= S^{\ell_1}L_f(\bcirc{\AA}, E_{IJ})(S^T)^{\ell_2},
\end{align*}
where for the second equality we have used the fact that powers of block circulant matrices are block circulant (and thus invariant under $\SS$), and the third equality follows from the fact that $S$ is unitary.
\end{proof}

As a special case, by choosing $\ell_1 = \ell_2$, Proposition~\ref{pro:shift_EIJ} states that the shift operator $\SS$ defined in~\eqref{eq:circulant_shift} can be ``pulled out'' of the Fr\'echet derivative, 
\begin{equation*}
L_f(\bcirc{\AA}, \SS^\ell(E_{IJ})) = \SS^\ell \left(L_f(\bcirc{\AA}, E_{IJ})\right).
\end{equation*}

In particular, choosing $\ell_1 = 0$ or $\ell_2 = 0$ (and denoting the other one simply by $\ell$), Proposition~\ref{pro:shift_EIJ} reveals that all Fr\'echet derivatives $L_f(\bcirc{\AA}, S^\ell E_{IJ})$ and $L_f(\bcirc{\AA}, E_{IJ}(S^T)^\ell)$ have exactly the same entries for any $\ell = 0, \ldots, p-1$, just shifted. It thus suffices to compute one of these Fr\'echet derivatives and then obtain the others essentially for free by applying $S$ and/or $S^T$. In total, it is enough to compute $L_f(\bcirc{\AA}, E_{IJ})$ for $I, J = 1,\ldots, n$, as all other canonical basis matrices $E_{IJ}$ can be generated by appropriate shifts.

\begin{remark}\label{rem:tubal_vectors}
For ``tubal vectors'' $\AA \in \spC^{1 \times 1 \times p}$, as they appear in certain tensor neural networks~\cite{MalUHetal19,NewHAetal18}, the preceding discussion implies that \emph{all} columns of $K_f(\AA) \in \spC^{p \times p}$ are shifted copies of the same vector. Thus, in this case, $K_f(\AA)$ is a circulant matrix.
\end{remark}

\section{Computing the t-Fr\'echet derivative} \label{sec:algorithms}
The primary challenge in computing with tensors is the so-called ``curse of dimensionality,"  to which the t-product formalism is not immune.  At the same time, due to the equivalence with functions of block circulant matrices, the tools at our disposal are largely limited by what has been developed for matrix functions in general.  We discuss viable approaches, along with potential tricks for reducing the overall complexity of computing the t-Fr\'echet derivative.

\subsection{A basic block Krylov subspace method} \label{sec:low_rank_krylov}
We recall from~\eqref{eq:proof_integral1} in the proof of Lemma~\ref{lem:integral_representation} that
\begin{equation} \label{eq:tfrechet_integral}
    L_f(\AA,\CC) = \fold{\frac{1}{2\pi i} \int_\Gamma f(\zeta) A_\zeta^{-1} C A_\zeta^{-1} \d \zeta \cdot \vE_1^{np \times n}},
\end{equation}
where $A_\zeta := \bcirc{\zeta \II - \AA}$ and $C := \bcirc{\CC}$. The integral term appearing in~\eqref{eq:tfrechet_integral} can be approximated by a block Krylov algorithm when the direction term $C$ is of low rank and can thus be written in the form $C = \vC_1 \vC_2^H$ with $\vC_1, \vC_2 \in \spC^{np \times r}, r \ll np$.

\begin{remark} \label{rem:rank_one}
    As an illustration, let us focus on the special case that $\CC$ is a rank-one tensor in the sense of the CP tensor format, i.e., that each entry fulfills
    \[
        \CC(i,j,k) = \vu(i)\cdot\vv(j)\cdot\vw(k), \qquad \vu, \vv \in \spC^{n}, \vw \in \spC^{p}.
    \]
    In this case, the $k$th frontal face of $\CC$ is of the form $C^{(k)} = \vw(k) \vu\vv^T$ and thus
    \begin{equation} \label{eq:bcircC_rankone}
        \bcirc{\CC} :=
        \begin{bmatrix}
        \vw(1)\vu\vv^T	& \vw(p)\vu\vv^T	& \vw(p-1)\vu\vv^T		& \cdots	& \vw(2)\vu\vv^T	\\
        \vw(2)\vu\vv^T	& \vw(1)\vu\vv^T	& \vw(p)\vu\vv^T		& \cdots	& \vw(3)\vu\vv^T	\\
        \vdots	& \ddots	& \ddots		& \ddots	& \vdots	\\
        \vw(p)\vu\vv^T	& \vw(p-1)\vu\vv^T	& \cdots		& \vw(2)\vu\vv^T & \vw(1)\vu\vv^T
        \end{bmatrix}.
    \end{equation}
    The matrix~\eqref{eq:bcircC_rankone} has rank at most $p$\footnote{
        Letting $W$ denote the circulant matrix of $\vw$, we have $\bcirc{\CC} = W \otimes \vu \vv^T$. As $\rnk(W \otimes \vu \vv^T = \rnk(W)\rnk(\vu \vv^T)$ and clearly $\rnk(W) \leq p$ and $\rnk(\vu \vv^T) \leq 1$, the assertion holds.},
        and the low rank factors can be given explicitly in terms of $\vu, \vv, \vw$.

    Of particular interest is the case in which all three vectors $\vu,\vv,\vw$ are canonical unit vectors, which arises, e.g., when measuring the sensitivity of $f(\AA)$ with respect to changes in one specific entry of $\AA$ \cite{DeJNetal22,Sch23b}.  Also interesting is when just two of the three vectors are unit vectors, which would occur when measuring the sensitivity with respect to changes in the same entry across all frontal, horizontal, or lateral slices of $\AA$.
\end{remark}

We define a block Krylov subspace as the block span
\[
    \spK_d(A,\vC) := \blkspn\{\vC, A \vC, \ldots, A^{d-1} \vC\} \subset \spC^{np \times r},
\]
where $d$ is a small positive integer denoting the iteration index.  For more details on the theory and implementation of block Krylov subspaces, see, e.g., \cite{Gut07, FroLS17}.

The Krylov subspace algorithm from~\cite{KanKRetal21,Kre19} for approximating
\begin{equation} \label{eq:integral_bcirc}
    \frac{1}{2\pi i} \int_\Gamma f(\zeta) A_\zeta^{-1} \vC_1 \vC_2^H A_\zeta^{-1}\d \zeta
\end{equation}
now proceeds by building orthonormal bases $\bVV_d, \bWW_d \in \spC^{np \times dr}$ of the two block Krylov subspaces $\spK_d(A, \vC_1)$ and $\spK_d(A^H, \vC_2)$, with $A := \bcirc{\AA}$, yielding the following block Arnoldi decompositions:
\begin{align*}
    A \bVV_d &= \bVV_d \GG_d + G_{d+1,d} \vV_{d+1} \vE_{d+1}^H \\
    A^H \bWW_d &= \bWW_d \HH_d + H_{d+1,d} \vW_{d+1} \vE_{d+1}^H.
\end{align*}
Both $\GG_d = \bVV_d^H A \bVV_d$ and $H_d = \bWW_d^H A^{H} \bWW_d$ are $dr \times dr$ block upper Hessenberg matrices. An approximation $\widetilde{L}_d$ of~\eqref{eq:integral_bcirc} is then extracted from the tensorized Krylov subspace $\spK_d(A^H, \vC_2) \otimes \spK_d(A, \vC_1)$ via
\[
    \widetilde{L}_d := \bVV_d X_d \bWW_d^H,
\]
where $X_d$ is the $dr \times dr$ upper right block of
\[
    f \left(
    \begin{bmatrix}
        \GG_d   & (\bVV_d^H \vC_1)(\bWW_d^H \vC_2)^H	\\
                & \HH_d^H
    \end{bmatrix}
    \right).
\]
In light of~\eqref{eq:tfrechet_integral}, the final approximation for the Fr\'echet derivative is then given by
\begin{equation*}
    L_f(\AA,\CC) \approx \widetilde{\LL}_d := \fold{\widetilde{L}_d \cdot \vE_1^{np \times n}}.
\end{equation*}

\subsection{Using the DFT to improve parallelism}
Consider again \eqref{eq:tfrechet_compact}, specifically the argument of $f$.  Thanks to \eqref{eq:dft} and Theorem~\ref{thm:tensor_t-func_props}(iii), we can write
\begin{equation}
    f\left(
        \begin{bmatrix}
            \bcirc{\AA} & \bcirc{\CC} \\  & \bcirc{\AA}
        \end{bmatrix}
    \right)  =
    \FF^H
    f\left(
        \begin{bmatrix}
            \DD^A & \DD^C \\  & \DD^A
        \end{bmatrix}
    \right)
    \FF\label{eq:f_blockdiagonal}
\end{equation}
with $\DD^A = \blkdiag(D^A_1, \ldots, D^A_p)$, $\DD^C = \blkdiag(D^C_1, \ldots, D^C_p)$, and
\begin{equation*}
\FF = \begin{bmatrix} F_p \otimes I_n  &  \\ & F_p \otimes I_n \end{bmatrix}.
\end{equation*}
Using~\eqref{eq:frechet_2x2}, we can rewrite~\eqref{eq:f_blockdiagonal} as
\begin{equation}
    f\left(
        \begin{bmatrix}
            \bcirc{\AA} & \bcirc{\CC} \\  & \bcirc{\AA}
        \end{bmatrix}
    \right)  =
    \FF^H
        \begin{bmatrix}
            f(\DD^A) & L_f(\DD^A,\DD^C) \\  & f(\DD^A)
        \end{bmatrix}
    \FF.\label{eq:f_blockdiagonal2}
\end{equation}
The following theorem, which can be seen as a Dalecki\u{\i}-Kre\u{\i}n-type result for block diagonal matrices, will be helpful.
\begin{theorem} \label{thm:blockdiagonal}
Let $A, C \in \spC^{np \times np}$ be block diagonal matrices with $n \times n$ blocks, $A = \blkdiag(A_1, \ldots, A_p)$, $C = \blkdiag(C_1, \ldots, C_p)$ and let $f$ be analytic on a region containing $\spec(A)$.

Then $L_f(A,C) = \blkdiag(L_1,\ldots,L_p)$ with
\begin{equation} \label{eq:L_diagonal_blocks}
    L_i = L_f(A_i,E_i),\quad i = 1,\ldots,p.
\end{equation}
\end{theorem}
\begin{proof}
When $A$ and $C$ are block diagonal, then for any $k \geq 1$, we have
\begin{equation} \label{eq:power_of_2x2_matrix}
\begin{bmatrix}
    A & C \\  & A
\end{bmatrix}^k =
\begin{bmatrix}
    A^k & M^{(k)}  \\  & A^k
\end{bmatrix}
\end{equation}
where $M^{(k)} = \blkdiag(M^{(k)}_1, \ldots, M^{(k)}_p)$ with
\begin{equation} \label{eq:M_diagonal_blocks}
M^{(k)}_i = \sum\limits_{j=1}^k A_i^{j-1}C_i A_i^{k-j}, \qquad i = 1,\ldots,p.
\end{equation}
Let
\[
f(z) = \sum\limits_{k=0}^\infty a_k z^k
\]
be the power series representation of the analytic function $f$. Then, by~\eqref{eq:power_of_2x2_matrix}--\eqref{eq:M_diagonal_blocks}, we have
\begin{equation} \label{eq:f_2x2_blockdiagonal}
f\left(\begin{bmatrix}
    A & C \\  & A
\end{bmatrix}\right) = 
\begin{bmatrix}
    f(A) & L  \\  & f(A)
\end{bmatrix},
\end{equation}
where $L = \blkdiag(L_1,\ldots,L_p)$ and
\begin{equation} \label{eq:L_diagonal_blocks_proof}
L_i = \sum\limits_{k=1}^\infty a_k M^{(k)}_i = \sum\limits_{k=1}^\infty a_k \sum\limits_{j=1}^k A_i^{j-1}C_i A_i^{k-j}.
\end{equation}
By~\cite[Eq.~(3.24)]{Hig08b}, the right-hand side of~\eqref{eq:L_diagonal_blocks_proof} coincides with $L_f(A_i,C_i)$ and by~\eqref{eq:frechet_2x2}, the matrix $L$ in~\eqref{eq:f_2x2_blockdiagonal} equals $L_f(A,C)$, thus completing the proof.
\end{proof}

\begin{corollary} \label{cor:daleckii-krein}
Let $\AA, \CC \in \spC^{n \times n \times p}$ and let $f$ be $2np-1$ times continuously differentiable on a region containing $\spec(\bcirc{\AA})$. Further, let 
\[(
F_p \otimes I_n)\bcirc{\AA}(F_p^H \otimes I_n) = \DD^A
\]
and
\[
(F_p \otimes I_n)\bcirc{\CC}(F_p^H \otimes I_n) = \DD^C
\]
with $\DD^A = \blkdiag(D^A_1, \ldots, D^A_p)$, $\DD^C = \blkdiag(D^C_1, \ldots, D^C_p)$. Then
\begin{equation} \label{eq:daleckii-krein_tfrechet}
L_f(\AA,\CC) = \fold{(F_p^H \otimes I_n)
\begin{bmatrix}
    \frac{1}{\sqrt{p}}L_1 \\
    \vdots \\
    \frac{1}{\sqrt{p}}L_p
\end{bmatrix}},
\end{equation}
where the diagonal blocks $L_i, i = 1,\ldots,p$ are given by
\begin{equation}\label{eq:diagL}
L_i = L_f(D^A_i,D^C_i), \qquad i = 1,\ldots,p.
\end{equation}
\end{corollary}
\begin{proof}
Under the assumptions of the theorem, the existence of the Fr\'echet derivative is guaranteed by Lemma~\ref{lem:tfrechet}. By combining~\eqref{eq:tfrechet_compact} with~\eqref{eq:f_blockdiagonal2}, we have
\begin{equation} \label{eq:proof_daleckii_krein_corollary1}
L_f(\AA,\CC) = \fold{
    \left[
    \FF^H
        \begin{bmatrix}
            f(\DD^A) & L_f(\DD^A,\DD^C) \\  & f(\DD^A)
        \end{bmatrix}
    \FF
    \cdot
    \begin{bmatrix}
        O_{np \times n} \\ I_n \\ O_{n(p-1) \times n}
    \end{bmatrix}
    \right]_{1:np,:}
}.
\end{equation}
According to Theorem~\ref{thm:blockdiagonal}, we have $L_f(\DD^A,\DD^C) = \blkdiag(L_1,\ldots,L_p)$ where the diagonal blocks are given by 
\[
L_i = L_f(D^A_i,D^C_i), \qquad i = 1,\ldots,p.
\]
Further, by the definition of $\FF$, it holds that
\[
\FF
\cdot
\begin{bmatrix}
    O_{np \times n} \\ I_n \\ O_{n(p-1) \times n}
\end{bmatrix}
= 
\begin{bmatrix}
    F_p \otimes I_n & \\ & F_p \otimes I_n
\end{bmatrix}
\cdot
\begin{bmatrix}
    O_{np \times n} \\ \ve^p_1 \otimes I_n
\end{bmatrix}
=
\begin{bmatrix}
    O_{np \times n} \\ F_p\ve^p_1 \otimes I_n.
\end{bmatrix}.
\]
We therefore have
\begin{align}
&\phantom{=} \FF^H
\begin{bmatrix}
    f(\DD^A) & L_f(\DD^A,\DD^C) \\  & f(\DD^A)
\end{bmatrix}
\FF
\cdot
\begin{bmatrix}
    O_{np \times n} \\ I_n \\ O_{n(p-1) \times n}
\end{bmatrix}\nonumber\\
& =
\FF^H
\begin{bmatrix}
    L_f(\DD^A,\DD^C)\cdot(F_p\ve^p_1 \otimes I_n) \\  f(\DD^A)\cdot(F_p\ve^p_1 \otimes I_n)
\end{bmatrix}\nonumber\\
& =
\begin{bmatrix}
    (F_p^H \otimes I_n)\cdot L_f(\DD^A,\DD^C)\cdot(F_p\ve^p_1 \otimes I_n) \\  (F_p^H \otimes I_n)\cdot f(\DD^A)\cdot(F_p\ve^p_1 \otimes I_n)
\end{bmatrix}.
\label{eq:proof_daleckii_krein_corollary2}
\end{align}
We now focus on the upper half of~\eqref{eq:proof_daleckii_krein_corollary2}, as only this block is needed for evaluating~\eqref{eq:proof_daleckii_krein_corollary1}. Due to the structure of $L_f(\DD^A,\DD^C)$, we have
\begin{align}
L_f(\DD^A,\DD^C)\cdot(F_p\ve^p_1 \otimes I_n) &= \blkdiag(L_1,\ldots,L_p)\cdot(F_p\ve^p_1 \otimes I_n)\nonumber\\
&= 
\begin{bmatrix}
    \frac{1}{\sqrt{p}}L_1 \\
    \vdots \\
    \frac{1}{\sqrt{p}}L_p
\end{bmatrix},\label{eq:proof_daleckii_krein_corollary3}
\end{align}
where we have used that the DFT matrix fulfills $F_p\ve_1^p = \frac{1}{\sqrt{p}}\vone$. Inserting~\eqref{eq:proof_daleckii_krein_corollary2} and~\eqref{eq:proof_daleckii_krein_corollary3} into~\eqref{eq:proof_daleckii_krein_corollary1} completes the proof.
\end{proof}

Corollary~\ref{cor:daleckii-krein} shows that by applying a DFT, the computation of the t-Fr\'echet derivative can be decoupled into the evaluation of $p$ Fr\'echet derivatives of $n \times n$ matrices that are completely independent of one another, thus giving rise to an embarrassingly parallel method. However, as the matrices $D_i^A, D_i^C$ occurring in~\eqref{eq:diagL} are in general dense and unstructured, computing these Fr\'echet derivatives is only feasible for moderate values of $n$ (but possibly large $p$).

\section{Applications of the t-Fr\'echet derivative} \label{sec:applications}
In this section, we briefly discuss two applications of the t-Fr\'echet formalism, namely condition number estimation for tensor functions and the gradient of the tensor nuclear norm.

\subsection{The condition number of the t-function} \label{sec:condition}
In practical applications, one often works with noisy or uncertain data, and additionally any computation in floating point arithmetic introduces rounding errors. Therefore, when working with the tensor t-function in practice, it is very important to understand how sensitive it is to perturbations in the data. This is measured by \emph{condition numbers}. 

The (absolute) condition number of the t-function can be defined by simply extending the well-known concept of condition number of scalar and matrix functions (see, e.g.,~\cite[Chapter~3]{Hig08b}), yielding
\begin{equation*}
    \condabs(f, \AA) := \lim_{\varepsilon \rightarrow 0} \sup_{\norm{\CC} \leq \varepsilon} \frac{\norm{f(\AA+\CC) - f(\AA)}}{\varepsilon},
\end{equation*}
where for our setting, $\norm{\cdot}$ denotes the norm~\eqref{eq:tensor_norm}, but can in principle also be any other tensor norm. A relative condition number can be readily defined as
\begin{equation*} 
    \condrel(f, \AA) := \lim_{\varepsilon \rightarrow 0} \sup_{\norm{\CC} \leq \varepsilon\norm{f(\AA)}} \frac{\norm{f(\AA+\CC) - f(\AA)}}{\varepsilon\norm{f(\AA)}} = \condabs(f,\AA) \frac{\norm{\AA}}{\norm{f(\AA)}}.
\end{equation*}

Completely analogously to the matrix function case, the condition number of the t-function can be related to the norm of its Fr\'echet derivative. 

\begin{lemma} \label{lem:condition_number_frechet}
    Let $f$ and $\AA$ be such that $L_f(\AA,\cdot)$ exists and denote
    \begin{equation} \label{eq:norm_L}
        \norm{L_f(\AA)} := \max_{\CC \neq 0} \frac{\norm{L_f(\AA,\CC)}}{\norm{\CC}}.
    \end{equation}
    Then the absolute and relative condition number of $f(\AA)$ are given by
    \begin{align*}
        \condabs(f,\AA) &= \norm{L_f(\AA)}, \\
        \condrel(f,\AA) &= \frac{\norm{L_f(\AA)}\norm{\AA}}{\norm{f(\AA)}}.
    \end{align*}
\end{lemma}
\begin{proof}
    The proof follows by using exactly the same line of argument as in the proof of~\cite[Theorem~3.1]{Hig08b} for the matrix function case, which only requires linearity of the Fr\'echet derivative and working in a finite-dimensional space and thus holds verbatim in our setting.
\end{proof}

Lemma~\ref{lem:condition_number_frechet} relates the condition number of the t-Fr\'echet derivative to the tensor-operator norm $\norm{L_f(\AA)}$, the computation of which might not be immediately clear (as the quantities on the right-hand side of~\eqref{eq:norm_L} are third-order tensors). The next result relates it to the spectral norm of the Kronecker form $\norm{K_f(\AA)}$.

\begin{lemma}
    Let $f$ and $\AA$ be such that $L_f(\AA,\cdot)$ exists and denote by $K_f(\AA)$ the Kronecker form of the Fr\'echet derivative, as defined in~\eqref{eq:kronecker_form_def}. Then
    \begin{equation}\label{eq:condition_number_kronecker}
    \norm{L_f(\AA)} = \norm{K_f(\AA)}_2.
    \end{equation}
\end{lemma}
\begin{proof}
    By the definition of the tensor norm~\eqref{eq:tensor_norm}, it is clear that $\norm{\BB} = \norm{\vec{\BB}}_2$ for any tensor $\BB$. Thus
    \[
    \norm{L_f(\AA)} = \max_{\CC \neq 0} \frac{\norm{\vec{L_f(\AA,\CC)}}_2}{\norm{\vec{\CC}}_2} = \max_{\CC \neq 0} \frac{\norm{K_f(\AA)\vec{\CC}}_2}{\norm{\vec{\CC}}_2} = \norm{K_f(\AA)}_2.
    \]
\end{proof}

For realistic problem sizes, it will typically not be feasible to compute the condition number of $f(\AA)$ via~\eqref{eq:condition_number_kronecker}. This is already the case for functions of $n \times n$ matrices, and it becomes even more prohibitive in the tensor setting. As outlined at the end of Section~\ref{sec:kronecker}, simply forming the Kronecker form $K_f(\AA)$ has cost $\mathcal{O}(n^5p^4)$ and requires $\mathcal{O}(n^4p^2)$ storage. Even for moderate values of $n$ and $p$, this is typically not possible.

Instead, we need to approximate the condition number. As a rough estimate is usually sufficient, a few steps of power iteration typically give a satisfactory result, as one is mainly interested in the order of magnitude of the condition number, so that more than one significant digit is seldom needed. Algorithm~\ref{alg:power_iteration} is a straightforward adaptation of~\cite[Algorithm~3.20]{Hig08b}, which computes an estimate of $\norm{K_f(A)}_2$ by applying power iteration to the Hermitian matrix $K_f(A)^HK_f(A)$, exploiting that a matrix vector multiplication $K_f(A)\vv$ is equivalent to the evaluation of $L_f(A,\unvec{\vv})$, where $\unvec{\vv}$ maps the vector $\vv$ to an unstacked matrix of the same size as $A$. In line~\ref{line:f_bar}, the function $\overline{f}$ is defined via $\overline{f}(z) = \overline{f(\overline{z})}$.

\begin{remark}
As Algorithm~\ref{alg:power_iteration} boils down to a matrix power iteration, its asymptotic convergence rate is linear and depends on the magnitude of the ratio between the eigenvalue of largest and second largest magnitude of the Hermitian matrix $K_f(\AA)^HK_f(\AA)$; see e.g.,~\cite[Eq.~(7.3.5)]{GolV13}. It is quite difficult, however, to give meaningful a priori bounds on this ratio, as we do not have explicit formulas for the eigenvalues or singular values of $K_f(\AA)$ available (in terms of spectral quantities related to $\AA$), and deriving such relations is well beyond the scope of this work.

Also, note that typically only $\mathcal{O}(1)$ iterations of Algorithm~\ref{alg:power_iteration} are sufficient due to the rather low accuracy requirements in condition number estimation; see our experiments reported in Section~\ref{subsec:experiments_cond} as well as, e.g.,~\cite{Hig08b, KenL89} for the matrix function case. In these early iterations, the asymptotic convergence rate will likely not be descriptive concerning the actual behavior of the method, as it does not capture the fast reduction of contributions from eigenvectors corresponding to small eigenvalues.
\end{remark}

\begin{algorithm}[t]
    \caption{\label{alg:power_iteration} Power iteration for the t-Fr\'echet derivative}
    \begin{algorithmic}[1]
        \setstretch{1.2}
        \smallskip
        
        \Statex \textbf{Input:}\ \ \ $f$, $\AA$, $\texttt{tol}$, $\texttt{max\_it}$
        \Statex \textbf{Output:} Estimate $\gamma \approx \norm{L(\AA)}$
        \State Choose $\CC_1 \in \spC^{n \times n \times p}$ at random
        
        \smallskip
        
        \For{$k = 1,\ldots,\texttt{max\_it}$}
            \State $\BB_{k+1} \leftarrow L_f(\AA, \CC_k)$
            \State $\CC_{k+1} \leftarrow L_{\overline{f}}(\AA^H, \BB_{k+1})$\label{line:f_bar}
            \State $\gamma_{k+1} \leftarrow \norm{\CC_{k+1}}/\norm{\BB_{k+1}}$
            \If{$\lvert\gamma_{k+1}-\gamma_k\rvert \leq \texttt{tol}\cdot\gamma_{k+1}$}
                \State break
            \EndIf
        \EndFor
        
        \smallskip
        
        \State $\gamma \leftarrow \gamma_{k+1}$
    \end{algorithmic}
\end{algorithm}

Algorithm~\ref{alg:power_iteration} is necessarily sequential with respect to calls of $L_f(\AA, \cdot)$.  An alternative algorithm that would lend itself naturally to parallelization (especially in the case that $n \ll p$) stems from Lemma~\ref{lem:column_relation} and Proposition~\ref{pro:shift_EIJ}, and is a variant implementation of Algorithm~\ref{alg:full_kronecker}.  In the first phase, $K_f(\bcirc{\AA})$ is computed but in a reduced fashion, whereby only $n^2$ applications of $L_f(\bcirc{\AA},\cdot)$ are required, thanks to the shift relation proven in Proposition~\ref{pro:E_ijk}.  This first step can be trivially parallelized, as it is known a priori exactly on which unit matrices to call $L_f(\bcirc{\AA},\cdot)$.  In the second phase, the columns of $K_f(\AA)$ are assembled via Lemma~\ref{lem:column_relation}.  While Algorithm~\ref{alg:full_kronecker} can similarly be trivially parallelized, the approach outlined in Algorithm~\ref{alg:efficient_kronecker} guarantees $n^2$ calls to $L_f(\bcirc{\AA}, \cdot)$ overall, as opposed to $n^2p$ in Algorithm~\ref{alg:full_kronecker}.

\begin{algorithm}[t]
\caption{\label{alg:efficient_kronecker} Kronecker form of the t-Fr\'echet derivative (efficient approach)}
    \begin{algorithmic}[1]
    \setstretch{1.2}
    \smallskip
    
    \Statex \textbf{Input:}\ \ \ $f$, $\AA$
    \Statex \textbf{Output:} Kronecker form $K = K_f(\AA)$
    \smallskip
    
    \State Allocate memory for $n^2$ matrices $\YY_{IJ} \in \spC^{np \times np}, I,J = 1,\dots,n$
    \For{$I = 1,\ldots,n$}
        \For{$J = 1,\ldots,n$}
            \State $\YY_{IJ} \leftarrow L_f(\bcirc{\AA}, E_{IJ})$, $E_{IJ} = \ve_1^T \otimes \unfold{\EE_{ijk}}$
        \EndFor
    \EndFor

    \For{$i = 1, \ldots, n$}
        \For{$j = 1, \ldots, n$}
            \For{$k = 1, \ldots, p$}
                \State $\XX_\ell \leftarrow \vec{\SS^{\ell-1}(\YY_{ij})}$, $\ell = 1,\dots,p$
                \State $K(:, i + (k-1)n + (j-1)np) \leftarrow \sum_{\ell = 1}^p  \XX_{\ell}(1:n^2p)$
            \EndFor
        \EndFor
    \EndFor    
    \end{algorithmic}
\end{algorithm}

We end this section by briefly discussing the connection between conditioning of the t-function $f(\AA)$ and the matrix function $f(\bcirc{\AA})$. In light of~\eqref{eq:norm_L} and the definition of $f(\AA)$ in terms of block circulant matrices, it is immediate that
\begin{equation}\label{eq:inequality_tensor_matrix_cond}
    \condabs(f,\AA) \leq \condabs(f,\bcirc{\AA}),
\end{equation}
where $\condabs(f,\bcirc{\AA})$ denotes the matrix function condition number in the Frobenius norm: the left-hand side of~\eqref{eq:inequality_tensor_matrix_cond}, when interpreted in terms of the underlying matrix function, only allows \emph{structured, block-circulant perturbations}, while the right-hand side measures conditioning with respect to \emph{any} perturbation. Often, such structured condition numbers can be significantly lower than unstructured condition numbers; see, e.g.,~\cite{ArsNT19,Dav04}. 
In our experiments, we have actually observed equality in~\eqref{eq:inequality_tensor_matrix_cond} in most test cases, at least up to machine precision, but it is also possible to construct examples in which the two condition numbers disagree by a large margin; see, e.g., the test script \texttt{test\_cond\_counter\_ex.m} in our code suite. It might be an interesting question for further research to find out whether there are conditions on $f$ and/or $\AA$ that guarantee equality holds in~\eqref{eq:inequality_tensor_matrix_cond}.

\subsection{The gradient of the tensor nuclear norm}\label{sec:tensor_nuclear_norm}
In this section, we highlight an example application of how our framework for the t-Fr\'echet derivative can be useful for deriving certain theoretical results in a rather straightforward fashion.

The \emph{nuclear norm} of a tensor is typically defined in terms of a tensor singular value decomposition (see, e.g.,~\cite{LuPW19}), but it was recently shown that it can also be computed in terms of the t-square root as
\begin{equation*}
\norm{\AA}_{\star} = \trace_{(1)}(\sqrt{\AA^T * \AA}),
\end{equation*}
where $\trace_{(1)}$ denotes the trace of the first frontal slice; see~\cite[Lemma~6]{BenEJetal22a}. Tensor nuclear norm minimization is an important tool in image completion, low-rank tensor completion, denoising, seismic data reconstruction, and principal component analysis; see, e.g., ~\cite{BenEJetal22,HosOM16,KreSS13,LiuZT19,LuFCetal20,LuPW19,YuaZ16,ZhaN19}.  In these applications, it can be of interest to compute the \emph{gradient} of the tensor nuclear norm for a gradient descent scheme.\footnote{We note that the tensor nuclear norm is clearly not differentiable at all tensors $\AA$, so one might also need to consider subgradients in certain applications, but this is well beyond the scope of this paper. We therefore only focus on the differentiable case here.} We will now derive an explicit formula for the gradient of $\norm{\AA}_\star$ in terms of t-functions, which is reminiscent of similar results in the matrix case. 

To do so, we first collect some auxiliary results on the $\trace_{(1)}$ operator. Clearly, $\trace_{(1)}$ is linear, and by direct computation, it is easy to verify that
\begin{equation*}
\norm{\AA} = \sqrt{\trace_{(1)}(\AA^T * \AA)},
\end{equation*}
where $\norm{\cdot}$ is the tensor norm defined in~\eqref{eq:tensor_norm} and that
\begin{equation}\label{eq:inner_product_trace1}
\langle \AA, \BB \rangle := \trace_{(1)}(\BB^T * \AA)
\end{equation}
defines an inner product on $\spC^{n \times n \times p}$ (which corresponds to the standard inner product on $\spC^{n^2p}$ for the vectorized tensors).

Further, the $\trace_{(1)}$ operator inherits the cyclic property of the trace, with respect to the t-product.
\begin{lemma}\label{lem:trace1_cyclic}
Let $\AA, \BB \in \spC^{n \times n \times p}$. Then
\[
\trace_{(1)}(\AA * \BB) = \trace_{(1)}(\BB * \AA).
\]
\end{lemma}
\begin{proof}
By the definition of the t-product $\AA * \BB := \fold{\bcirc{\AA} \unfold{\BB}}$, the first face of $\AA * \BB$ is the first $n \times n$ block of $\bcirc{\AA} \unfold{\BB}$, which is given by
\begin{equation}\label{eq:proof_cyclic1}
[\bcirc{\AA} \unfold{\BB}]_{1:n,:} = A^{(1)}B^{(1)} + A^{(p)}B^{(2)} + \dots + A^{(2)}B^{(p)}.
\end{equation}
Similarly, the first face of $\BB * \AA$ is 
\begin{equation}\label{eq:proof_cyclic2}
[\bcirc{\BB} \unfold{\AA}]_{1:n,:} = B^{(1)}A^{(1)} + B^{(p)}A^{(2)} + \dots + B^{(2)}A^{(p)}.
\end{equation}
Using the linearity and the cyclic property of the trace, it is clear that the traces of~\eqref{eq:proof_cyclic1} and~\eqref{eq:proof_cyclic2} agree, thus proving the result of the lemma.
\end{proof}

Lemma~\ref{lem:trace1_cyclic} together with Lemma~\ref{lem:integral_representation} leads to a useful representation for the derivative of $\trace_{(1)}(f(\AA))$ when $f$ is analytic, involving the derivative of the scalar function $f$. By a slight abuse of notation, we write the Fr\'echet derivative (in the sense of the general definition~\eqref{eq:frechet_general}) of $\trace_{(1)}$ at a tensor $\MM$ as $L_{\trace_{(1)}}(\MM,\cdot)$, although it is clearly not a t-function. 

\begin{lemma}\label{lem:trace1_derivative}
Let $\AA \in \spC^{n \times n \times p}$ and let $f$ be analytic on a region containing the spectrum of $\bcirc{\AA}$. Then
\[
L_{\trace_{(1)} \circ f}(\AA,\CC) = \trace_{(1)}(f^\prime(\AA) * \CC).
\]
\end{lemma}
\begin{proof}
By the linearity of $\trace_{(1)}$ we directly obtain
\[
L_{\trace_{(1)}}(\AA,\CC) = \trace_{(1)}(\CC).
\]
As the chain rule, Lemma~\ref{lem:properties_tfrechet}\emph{(\ref{item:chainrule})}, also holds more generally for any Fr\'echet differentiable functions, not necessarily t-functions, we have
\begin{equation} \label{eq:trace1_composition}
    L_{\trace_{(1)} \circ f}(\MM,\CC) = L_{\trace_{(1)}}(f(\MM),L_f(\MM,\CC)) = \trace_{(1)}(L_f(\MM,\CC)).
\end{equation}

By Lemma~\ref{lem:integral_representation}, we can further rewrite \eqref{eq:trace1_composition} as
\begin{align}
L_{\trace_{(1)} \circ f}(\MM,\CC) &= \trace_{(1)}\left(\frac{1}{2\pi i} \int_\Gamma f(\zeta)(\zeta \II - \AA)^{-1} * \CC * (\zeta \II - \AA)^{-1}\d \zeta\right).\nonumber\\
                                  &= \trace_{(1)}\left(\frac{1}{2\pi i} \int_\Gamma f(\zeta)(\zeta \II - \AA)^{-2} \d \zeta * \CC \right).\label{eq:trace_derivative}
\end{align}
where we have used the cyclic property of $\trace_{(1)}$ with respect to the t-product from Lemma~\ref{lem:trace1_cyclic} for the second equality. The integral in~\eqref{eq:trace_derivative} is the Cauchy integral representation of $f^\prime(\AA)$, thus completing the proof.
\end{proof}

We are now in a position to state the main result of this section. Note that using the inner product~\eqref{eq:inner_product_trace1}, the gradient of the nuclear norm can be characterized by imposing the condition
\begin{equation}\label{eq:gradient_condition}
    L_{\norm{\cdot}_{\star}}(\AA, \CC) = \langle \CC, \nabla_\AA \norm{\AA}_\star \rangle = \trace_{(1)}{(\nabla_\AA \norm{\AA}_\star)^T * \CC},
\end{equation}
for all $\CC \in \spC^{n \times n \times p}$.

\begin{theorem}\label{the:nuclear_norm}
    Let $\AA \in \spC^{n \times n \times p}$ be such that $(\AA^T * \AA)^{-1/2}$ is defined. Then $\norm{\cdot}_{\star}$ is differentiable at $\AA$ and
    \[
    \nabla_A \norm{\AA}_{\star} = \AA * (\AA^T * \AA)^{-1/2}.
    \]
\end{theorem}
\begin{proof}
    Define $f(\MM) = \MM^T*\MM$, $g(\MM) = \sqrt{\MM}$, so that $\norm{\AA}_\star = (\trace_{(1)} \circ g \circ f)(\AA)$, where $f$ is \emph{not} a tensor t-function in the usual sense. As before, with slight abuse of notation, we write $L_f(\MM,\cdot)$ for its Fr\'echet derivative. From the definition of the t-product, it is straightforward to verify that
    \begin{equation}\label{eq:proof_nuclear_norm_derivative_f}
        L_f(\AA, \CC) = \AA^T * \CC + \CC^T * \AA.
    \end{equation}
    Using the chain rule and Lemma~\ref{lem:trace1_derivative}, we have
        \begin{equation}\label{eq:proof_nuclear_norm2}
        L_{\trace_{(1)} \circ g \circ f}(\AA, \CC) = L_{\trace_{(1)} \circ g}(f(\AA), L_f(\AA,\CC)) = \trace_{(1)}(g^\prime(f(\AA))*L_f(\AA,\CC)).
    \end{equation}
    As $g$ is the square root, we have $g^\prime(f(\AA)) = \frac{1}{2}(\AA^T * \AA)^{-1/2}$, so that by combining~\eqref{eq:proof_nuclear_norm_derivative_f} and~\eqref{eq:proof_nuclear_norm2}, we find
    \begin{align}
       L_{\norm{\cdot}_\star}(\AA,\CC) = & L_{\trace_{(1)} \circ g \circ f}(\AA, \CC) \nonumber\\
        = &\frac{1}{2}\trace_{(1)}((\AA^T*\AA)^{-1/2}*\AA^T * \CC + (\AA^T\AA)^{-1/2}*\CC^T * \AA) \nonumber\\
        = &\frac{1}{2}\trace_{(1)}((\AA^T*\AA)^{-1/2}*\AA^T * \CC) + \frac{1}{2}\trace_{(1)}(\CC^T*\AA*(\AA^T*\AA)^{-1/2}) \nonumber\\
        = &\trace_{(1)}((\AA^T*\AA)^{-1/2}*\AA^T*\CC),\label{eq:proof_nuclear_norm3}
    \end{align}
    where we have used the cyclic property of $\trace_{(1)}$ for the second equality and the fact that $\trace_{(1)}{(\MM^T)} = \trace_{(1)}(\MM)$, which directly follows from the definition of tensor t-transposition, together with the linearity of $\trace_{(1)}$ for the third equality. Comparing~\eqref{eq:proof_nuclear_norm3} and~\eqref{eq:gradient_condition} shows that
    \[
    \nabla_A \norm{\AA}_{\star} = \AA * (\AA^T * \AA)^{-1/2},
    \]
    thus concluding the proof.
\end{proof}

To illustrate the theory, the script \texttt{test\_t\_nuclear\_norm.m} in our code suite implements a simple gradient descent scheme with backtracking line search for nuclear norm minimization, based on Theorem~\ref{the:nuclear_norm}.

\section{Numerical experiments} \label{sec:experiments}
In this section, we detail a software framework for studying the performance of the proposed algorithms and present numerical results from several small- to medium-scale experiments.

\subsection{Implementation details}
We have developed our own modular toolbox, \texttt{t-Frechet}, hosted at \url{https://gitlab.com/katlund/t-frechet}.  The basic syntax is derived from \bfomfom\footnote{\url{https://gitlab.com/katlund/bfomfom-main}} and \texttt{LowSyncBlockArnoldi}\footnote{\url{https://gitlab.mpi-magdeburg.mpg.de/lund/low-sync-block-arnoldi}}.  We note that in contrast to an existing t-product toolbox \texttt{Tensor-tensor-product-toolbox}\footnote{\url{https://github.com/canyilu/Tensor-tensor-product-toolbox}}, a tensor $\AA$ in \texttt{t-Frechet} is encoded as a MATLAB struct with fields \texttt{mat} and \texttt{dim}, which store $\unfold{\AA}$ and $\AA$'s dimensions as a vector $[n\, m\, p]$, respectively.  Such tensor structs allow us to work with sparse tensors via built-in MATLAB functions and compute the actions of block circulant matrices without ever explicitly forming the full $np \times mp$ matrix.  Our toolbox has been tested in MATLAB 2019b, 2022a, and 2023a on Ubuntu and Windows machines.

Table~\ref{tab:t_frechet_approaches} summarizes features of the three methods for approximating $L_f(\AA,\CC)$ that we have derived throughout the text.  Regarding the \dft approach, note that equation~\eqref{eq:daleckii-krein_tfrechet} can be trivially implemented on (dense) third-order arrays in MATLAB, thanks to \texttt{fft} and \texttt{ifft}; see comments in \cite{KilBHetal13} as well as our test script \texttt{test\_dft}.  A number of additional test scripts are included in \texttt{t-Frechet} that we do not discuss here; we have, however, kept them public to encourage further engagement with the community.
\begin{table}[H]
    \centering
    \begin{tabular}{c|c|M{.1\textwidth}|M{.07\textwidth}|M{.07\textwidth}|M{.11\textwidth}|M{.085\textwidth}}
        \thead{Approach}    & \thead{Operator\\(Op.)}  & \thead{Op.~size}  & \thead{No.~of\\Op.}& \thead{Sparse\\op.?} & \thead{Transpose\\required?}  & \thead{Restarts\\allowed?} \\ \hline
        \bcirccode, \eqref{eq:tfrechet_compact}         & $\bcirc{\AA}$   & $np \times np$ & 1  & Y & N & Y \\
        \lowrank, Sec.~\ref{sec:low_rank_krylov}  & $\bcirc{\AA}$   & $np \times np$ & 2  & Y & Y & N \\
        \dft, Cor.~\ref{cor:daleckii-krein}         & $D_{\AA}$       & $n \times n$   & p  & N & N & Y
    \end{tabular}
    \caption{Features of numerical approaches for computing $L_f(\AA,\CC)$.  Note that for \lowrank, the number of operators refers to the fact that the transpose is needed, which is nontrivial if $\AA$ is only known implicitly or via a black-box routine.  As for \dft, $D_{\AA}$ represents all $p$ subproblems.}
    \label{tab:t_frechet_approaches}
\end{table}

\subsection{Comparing performance of t-Fr\'echet implementations} \label{sec:performance_t_frechet}
We consider a simple example for examining the performance of the proposed solvers by taking $f(z) = \exp(z)$ and $\AA \in \spC^{n \times n \times p}$ such that each face of $\AA$ is a finite differences stencil for the spatial components of the two-dimensional convection-diffusion equation
\[
u_t = -\Delta (u_{xx} + u_{yy}) + \nu (u_x + u_y)
\]

with the convection parameter $\nu$ drawn $p$ times uniformly from the interval $[0, 200]$.  We restrict both spatial variables to the unit square and take $\sqrt{n}$ points in each direction, where $n \in \{36, 144, 576\}$.  The direction tensor $\CC$ is dense and its entries are randomly drawn from the normal distribution.

All scripts are executed in MATLAB R2022a on 16 threads of a single, standard node of the Linux Cluster Mechthild at the Max Planck Institute for Dynamics of Complex Technical Systems in Magdeburg, Germany.\footnote{A standard node comprises 2 Intel Xeon Silver 4110 (Skylake) CPUs with 8 Cores each (64KB L1 cache, 1024KB L2 cache), a clockrate of 2.1 GHz (3.0 GHz max), and 12MB shared L3 cache each.}  We report the total run time to reach a tolerance of $10^{-6}$, percentage speed-up, number of times the operator (see Table~\ref{tab:t_frechet_approaches}) is called, and the final error for all three approaches.  Each approach is run $10$ times, and the reported times are an average over these runs.  Unless otherwise mentioned, B(FOM)$^2$ \cite{FroLS17} with the classical inner product and block modified Gram-Schmidt was employed to compute the matrix functions.  Note that aside from node-level multithreading, all algorithms are run in serial.

\subsubsection{Small problem: $n = 36$, $p = 10$}
    The performance is similar for all algorithms for this small problem size, which leads to matrix function problems of size $360 \times 360$ for \bcirccode and \lowrank, and $36 \times 36$ for \dft.  However, both \lowrank and \dft converge very quickly---1 and 2 iterations, respectively---and achieve high accuracy.  Recall that both the \lowrank and \dft approaches rely on multiple operators per iteration.  Accuracy for \dft is measured as an average across all subproblems.  See Table~\ref{tab:t_exp_conf_diff_n36_p10} for performance data and Figure~\ref{fig:t_exp_conv_diff_n36_p10_error} for error plots of \bcirccode.
    \begin{table}[H] \label{tab:t_exp_conf_diff_n36_p10}
    	\begin{center}
		\begin{tabular}{l|c|c|c|c}
			\toprule
			\thead{Configuration}	& \thead{Time (s)}	& \thead{\%\\Speed-up}	& \thead{Op.\\count} & \thead{Final\\error}\\
			\midrule
			\bcirccode       & 0.41       & 0.00      & 13        & 6.1252e-07       \\\hline
			\lowrank         & 0.19       & 53.56     & 2         & 4.4444e-15       \\\hline
			\dft             & 0.14       & 65.33     & 20        & 2.5940e-15       
		\end{tabular}
    	\end{center}
    \end{table}

    \begin{figure}[H] \label{fig:t_exp_conv_diff_n36_p10_error}
        \centering
        \resizebox{.45\textwidth}{!}{\includegraphics{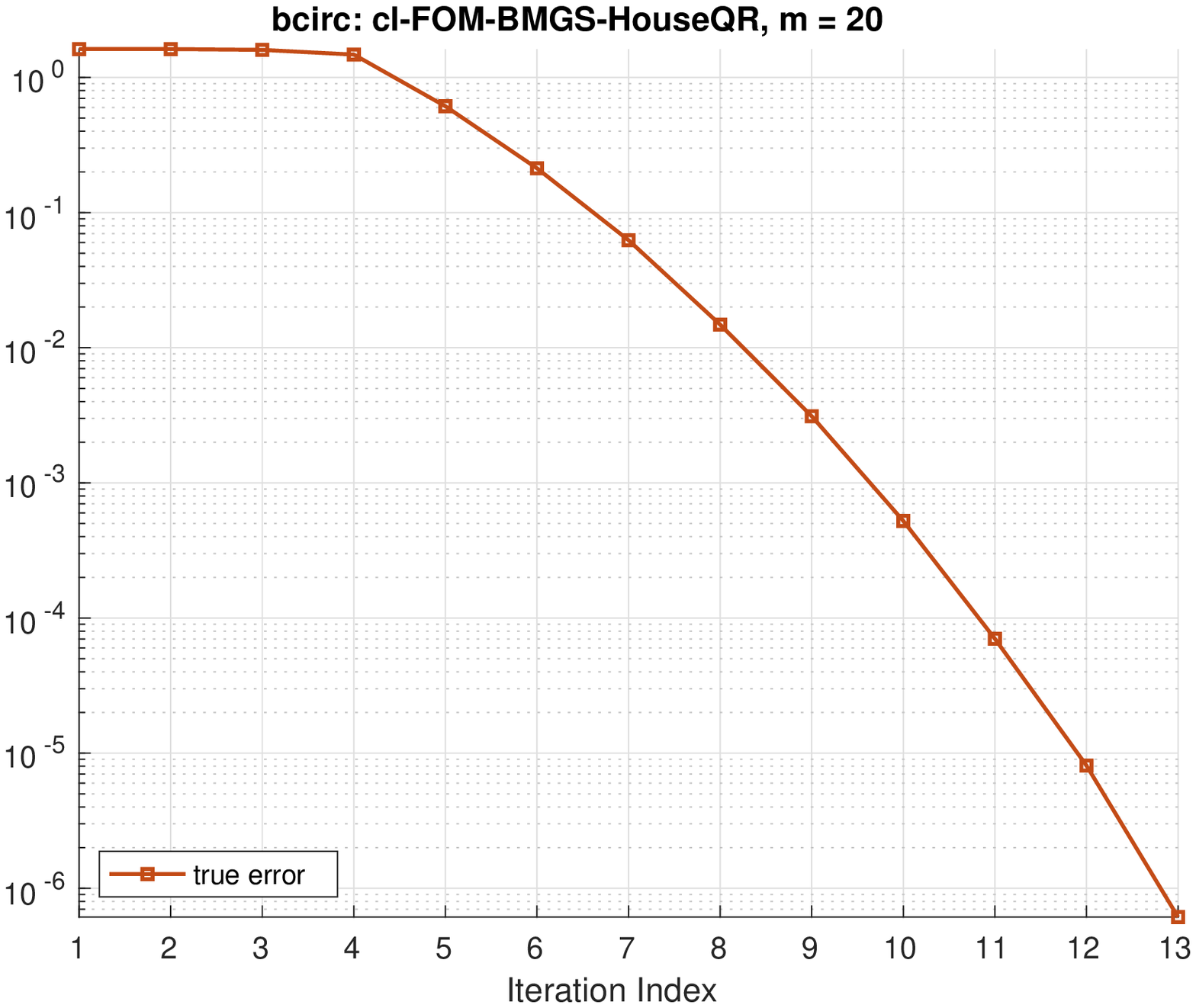}}
      \end{figure}

\subsubsection{Medium problem: $n = 144$, $p = 10$}
    With a larger problem size we begin to see clear performance differences among the three methods.  Matrix function problems are now $1440 \times 1440$ for \bcirccode and \lowrank, and $144 \times 144$ for \dft.  Both \bcirccode and \lowrank struggle to compete with \dft, which is an order of magnitude faster, due to computing with much smaller matrices.  Furthermore, \dft has no apparent accuracy issues, achieving near machine precision in 2 iterations, while \lowrank achieves a similar accuracy in 1 iteration and \bcirccode just passes the desired tolerance after 14 iterations.  See Table~\ref{tab:t_exp_conf_diff_n144_p10} for performance data and Figure~\ref{fig:t_exp_conv_diff_n144_p10_error} for error plots of \bcirccode.
    \begin{table}[H] \label{tab:t_exp_conf_diff_n144_p10}
    	\begin{center}
		\begin{tabular}{l|c|c|c|c}
			\toprule
			\thead{Configuration}	& \thead{Time (s)}	& \thead{\%\\Speed-up}	& \thead{Op.\\count} & \thead{Final\\error}\\
			\midrule
			\bcirccode       & 6.78       & 0.00       & 14          & 4.3093e-07       \\\hline
			\lowrank         & 3.94       & 41.91      & 2           & 6.8581e-15       \\\hline
			\dft             & 0.70       & 89.74      & 20          & 7.3250e-15       
		\end{tabular}
    	\end{center}
    \end{table}
    
    \begin{figure}[H] \label{fig:t_exp_conv_diff_n144_p10_error}
        \centering
        \resizebox{.45\textwidth}{!}{\includegraphics{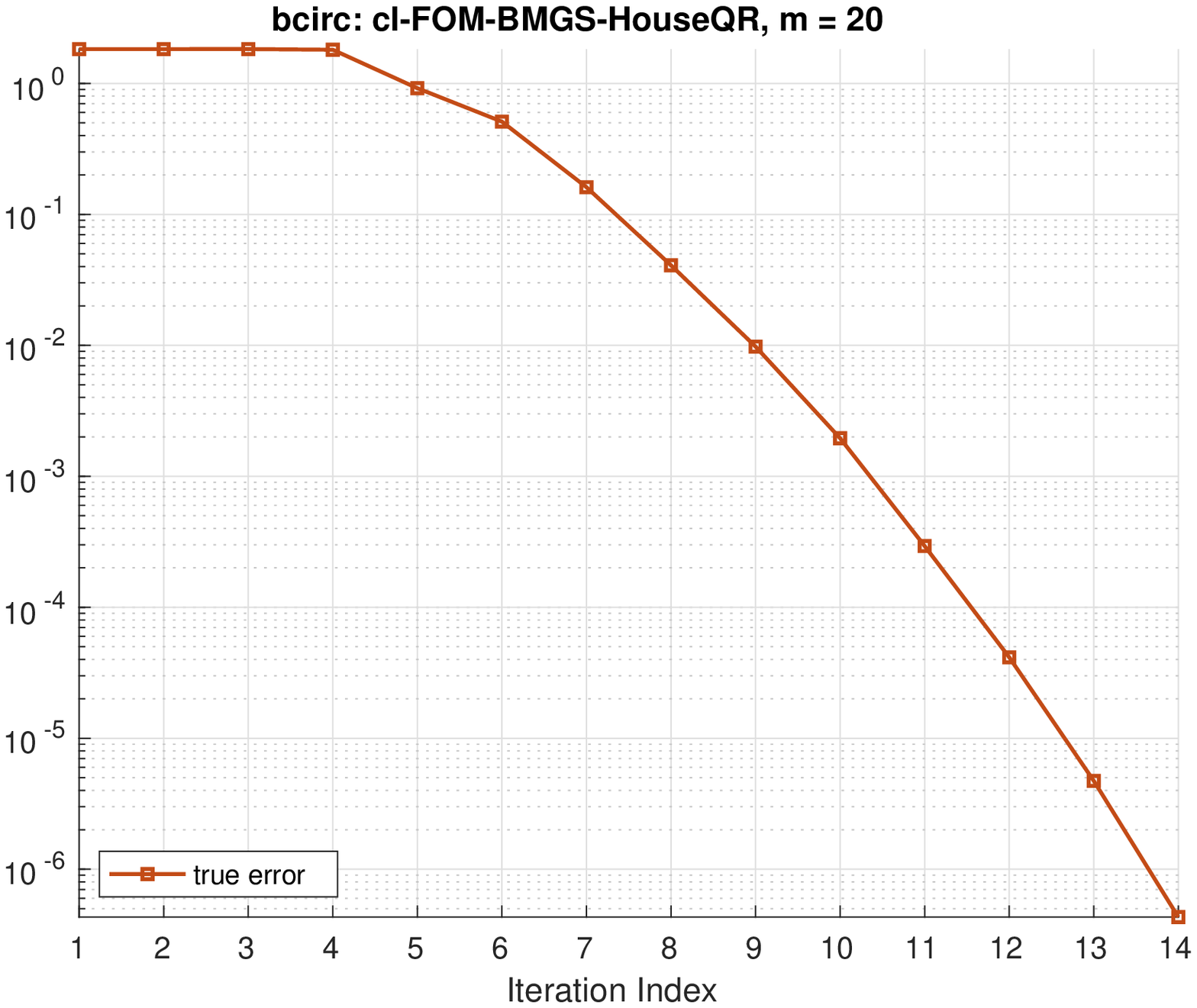}}
    \end{figure}

\subsubsection{Large problem: $n = 576$, $p = 10$}
As we quadruple the problem size, the situation remains nearly identical to when $n = 144$.  The \dft approach remains significantly faster than either \bcirccode, which still struggles to achieve better accuracy, and \lowrank, which despite requiring only 1 iteration is overall as slow as \bcirccode.  See Table~\ref{tab:t_exp_conv_diff_n576_p10} for performance data and Figure~\ref{fig:t_exp_conv_diff_n576_p10_error} for error plots of \bcirccode.  Note that due to the longer run time for this problem, we averaged timings over 5 instead of 10 runs.
\begin{table}[H] \label{tab:t_exp_conv_diff_n576_p10}
	\begin{center}
		\begin{tabular}{l|c|c|c|c}
			\toprule
			\thead{Configuration}	& \thead{Time (s)}	& \thead{\%\\Speed-up}	& \thead{Op.\\count} & \thead{Final\\error}\\
			\midrule
			\bcirccode      & 262       & 0.00       & 14          & 9.0656e-07       \\\hline
			\lowrank        & 204       & 22.02      & 2           & 2.2185e-14       \\\hline
			\dft            & 12.1      & 95.40      & 20          & 1.5919e-14       
		\end{tabular}
	\end{center}
\end{table}

\begin{figure}[H] \label{fig:t_exp_conv_diff_n576_p10_error}
    \centering
    \resizebox{.45\textwidth}{!}{\includegraphics{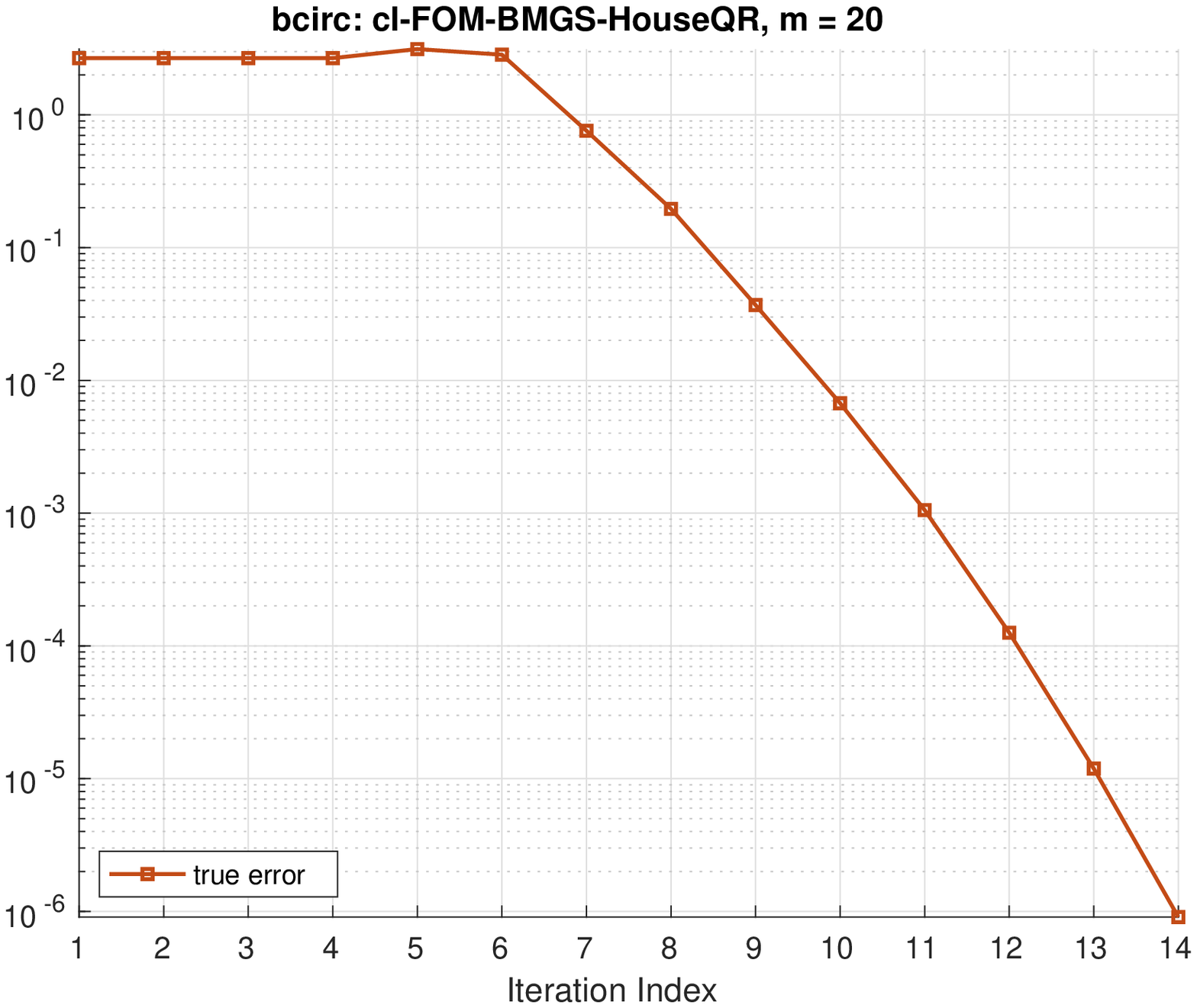}}
\end{figure}

\subsection{Accuracy and effort of t-condition number solvers}\label{subsec:experiments_cond}
For testing condition number algorithms, we fix the t-Fr\'echet solver to be an ``exact" (non-iterative) method.  We then study how different approaches fare with respect to the number of times they invoke a t-Fr\'echet solver, simply denoted as \tfrechet.  We take $f(z) = \exp(z)$ and $\AA$ a dense $n \times n \times p$ tensor, whose entries are drawn randomly from the normal distribution.  We set a tolerance of $10^{-2}$ for the power iteration, and we compare it with the ``full" Kronecker form approach (Algorithm~\ref{alg:full_kronecker}), which we also treat as ground truth, and the ``efficient" Kronecker form approach (Algorithm~\ref{alg:efficient_kronecker}).

For all the tests in this section, we only look at a single run, as computing the full Kronecker form is time-consuming.

\subsubsection{Big faces: $n = 20, p = 5$}
For the first example, we consider the case where $n > p$.  Results are summarized in Table~\ref{tab:t_func_cond_big_faces}.  The power iteration is clearly the winning method here, with only 8 calls to \tfrechet necessary to achieve the desired tolerance.  While the efficient Kronecker approach does reduce the overall time in comparison to the full Kronecker approach, it is not competitive with the power iteration.

\begin{table}[H] \label{tab:t_func_cond_big_faces}
	\begin{center}
		\begin{tabular}{l|c|c|c|c}
			\toprule
			\thead{Method}	& \thead{Time (s)} & \thead{\tfrechet\\calls} & \thead{Time (s)\\per call} & \thead{Accuracy}\\
			\midrule
			Power iteration      & 0.03       & 8          & 3.43e-02       & 3.3923e-03       \\\hline
			Efficient Kronecker  & 7.95       & 400        & 1.99e-02       & 4.3122e-16       \\\hline
			Full Kronecker       & 31.8       & 2000       & 1.59e-02       & 0.0000e+00       
		\end{tabular}
	\end{center}
\end{table}

\subsubsection{All things equal: $n = 10, p = 10$.}
We now examine the scenario where $n = p$.  Results are found in Table~\ref{tab:t_func_cond_all_equal}.  The power iteration remains significantly faster than both Kronecker form competitors, and it still achieves the desired tolerance.

\begin{table}[H] \label{tab:t_func_cond_all_equal}
	\begin{center}
		\begin{tabular}{l|c|c|c|c}
			\toprule
			\thead{Method}	& \thead{Time (s)} & \thead{\tfrechet\\calls} & \thead{Time (s)\\per call} & \thead{Accuracy}\\
			\midrule
			Power iteration      & 0.01       & 6          & 1.57e-02       & 3.1937e-03       \\\hline
			Efficient Kronecker  & 1.16       & 100        & 1.16e-02       & 3.6995e-16       \\\hline
			Full Kronecker       & 12.1       & 1000       & 1.21e-02       & 0.0000e+00       
		\end{tabular}
	\end{center}
\end{table}

\subsubsection{Many faces: $n = 5$, $p = 50$}
We finally consider $n \ll p$; see Table~\ref{tab:t_func_cond_many_faces} for the results.  The power iteration remains overwhelmingly faster than the efficient Kronecker approach, and still achieves the desired tolerance.

\begin{table}[H] \label{tab:t_func_cond_many_faces}
	\begin{center}
		\begin{tabular}{l|c|c|c|c}
			\toprule
			\thead{Method}	& \thead{Time (s)} & \thead{\tfrechet\\calls} & \thead{Time (s)\\per call} & \thead{Accuracy}\\
			\midrule
			Power iteration      & 1.33       & 6           & 2.21e-01       & 1.3325e-06       \\\hline
			Efficient Kronecker  & 16.5       & 25          & 6.60e-01       & 0.0000e+00       \\\hline
			Full Kronecker       & 189        & 1250        & 1.51e-01       & 0.0000e+00       
		\end{tabular}
	\end{center}
\end{table}

A clear drawback of the analysis in this section is that, in practice, one will not be able to compute Fr\'echet derivatives with high accuracy.  However, in most applications that require a condition number, accuracy is unimportant.  In which case it is sufficient to replace the inner \tfrechet solves of the power iteration with, for example, the \dft approach from Corollary~\ref{cor:daleckii-krein}.

When accuracy is important, however, the efficient Kronecker approach may be a viable competitor to the power iteration.  In all examples, we see that the time per \tfrechet evaluation is roughly the same per method.  Because all the \tfrechet problems are known a priori and they are far fewer than in the full Kronecker approach, the efficient Kronecker procedure is trivially parallelizable, unlike the power iteration, which is necessarily serial.  In the case with many faces (i.e., $n < p$), where relatively few \tfrechet calls overall are necessary, a simple parallelization could easily give the efficient Kronecker approach an edge.

\section{Conclusions} \label{sec:conclusion}
Thanks to the block circulant structure imposed by the t-product formalism, we have been able to take advantage of a rich mathematical framework not only in the definition of the Fr\'echet derivative of the tensor t-function but also in the development of efficient and accurate algorithms for its numerical approximation.  We have proven a number of useful properties of the t-Fr\'echet derivative, including a Dalecki\u{\i}-Kre\u{\i}n-type result.  An expression for the gradient of the nuclear norm has also been derived and its utility demonstrated in a gradient descent scheme for nuclear norm minimization.  We have affirmed the indispensability of the discrete Fourier transform (DFT) in accelerating the computation of the t-Fr\'echet derivative itself, as the DFT decouples the problem into $p$ smaller problems that each converge in few iterations.  We have further shown the utility of the t-Fr\'echet derivative in t-function condition number estimation.  A tailored power iteration algorithm has proven efficient for reliably computing the condition number at a high tolerance.  We have also demonstrated that the full Kronecker form of the t-Fr\'echet derivative can be computed in $p$ times less work than a direct approach thanks to symmetries evoked by the block circulant structure.  Finally, we have developed and made public a modular t-product toolbox that will prove foundational in exploring further, more challenging applications.




\addcontentsline{toc}{section}{References}
\bibliographystyle{plainurl}
\bibliography{tfrechet}
  
\end{document}